\documentclass[a4paper,12pt]{amsart}
\usepackage[a4paper,margin=1in]{geometry}%this is for easy margins
\usepackage{amsmath,verbatim}
\usepackage{enumerate}
\usepackage[T1]{fontenc}
\usepackage{eucal,mathrsfs,dsfont}%gives nicer set-names. Use \mathds instead of \mathds
\usepackage{color}%cyan,red;magenta,green;yellow,blue
\usepackage{marginnote}%defines nice marginal notes
\newcommand{\real}{\mathds{R}}
\newcommand{\rn}{{\mathds{R}^n}}
\newcommand{\rno}{{\mathds{R}^n\setminus\{0\}}}
\newcommand{\comp}{\mathds{C}}
\newcommand{\supp}{\operatorname{supp}}

\newcommand{\Ee}{\mathds E}
\newcommand{\Pp}{\mathds P}
\newcommand{\I}{\mathds 1}

\newcommand{\Kato}{\mathcal{S}_K}

\usepackage{color}%for comments

\newcommand{\normal}{\color{black}}

\numberwithin{equation}{section}

\theoremstyle{plain}
\newtheorem{theorem}{Theorem}
\newtheorem{lemma}[theorem]{Lemma}
\newtheorem{proposition}[theorem]{Proposition}
\newtheorem{corollary}[theorem]{Corollary}

\theoremstyle{definition}
\newtheorem{remark}[theorem]{Remark}
\newtheorem{definition}[theorem]{Definition}
\newtheorem{example}[theorem]{Example}
\newtheorem*{assumption-A}{Assumption~A}

\numberwithin{theorem}{section}

\usepackage{hyperref}

\begin{document}

\allowdisplaybreaks

\title[Level and collision sets of Feller processes]{\bfseries On level and collision sets of some Feller processes}

\author[V.~Knopova]{Victoria Knopova}
\address[V.~Knopova]{V.\,M.\ Glushkov Institute of Cybernetics\\NAS of Ukraine\\30187, Kiev, Ukraine}
%\curraddr{TU Dresden\\Fachrichtung Mathematik\\Institut f\"{u}r Mathematische Stochastik\\01062 Dresden, Germany}
\email{vic\underline{ }knopova@gmx.de}
\thanks{Financial support through the Scholarship for Young Scientists 2012--2014, Ukraine (for Victoria Knopova) and DFG (grant Schi 419/8-1)
(for Ren\'{e} L.\ Schilling) is gratefully acknowledged.}

\author[R.\,L.~Schilling]{Ren\'e L.\ Schilling}
\address[R.\,L.~Schilling]{TU Dresden\\Fachrichtung Mathematik\\Institut f\"{u}r Mathematische Stochastik\\01062 Dresden, Germany}
\email{rene.schilling@tu-dresden.de}

\subjclass[2010]{Primary 60G17. Secondary: 60J75; 60J25; 28A78; 35S05.}
\keywords{Feller process, pseudo-differential operator, symbol, Hausdorff dimension, image sets, level sets, regular points, polar sets, Kato class}

\date{\today}

\begin{abstract}
    This paper is about lower and upper bounds for the Hausdorff dimension of the level and collision sets of a class of Feller processes. Our approach is motivated by analogous results for L\'evy processes by Hawkes \cite{Ha74} (for level sets) and Taylor \cite{Ta66} and Jain \& Pruitt \cite{JP69} (for collision sets).  Since Feller processes lack independent or stationary increments, the methods developed for L\'evy processes cannot be used in a straightforward manner. Under the assumption that the Feller process possesses a transition probability density, which admits lower and upper bounds of a certain type, we derive sufficient conditions for regularity and non-polarity of points; together with suitable time changes this allows us to get upper and lower bounds for the Hausdorff dimension.
\end{abstract}

\maketitle

\section{Introduction}\label{intro}

In this paper, we study the Hausdorff dimension of the level and collision sets of a certain class of strong Feller processes; concrete examples were constructed in \cite{KK13a} and \cite{KK13b} under rather general assumptions, see Assumption~A below. This assumption guarantees, in particular, that the process is a strong Feller process admitting a transition probability density which enjoys upper and lower estimates of ``compound kernel'' type, see \eqref{set-e35} and \eqref{set-e36}.

Let us briefly describe the problems which are discussed in this paper. Let $X$ be a (strong) Feller process with values in $\rn$. Then
\begin{equation}\label{intro-e05}
    \left\{s \,:\, X_s(\omega)\in D\right\}
    \quad\text{for any Borel set $D\subset\rn$}
\end{equation}
denotes a level set of $X$, i.e.\ the (random) set of times when $X$ visits the set $D$.

We adapt the techniques from \cite{Ha74}, see also \cite{Ha70} and \cite{Ha71}, to obtain bounds on the Hausdorff dimension of such level sets. The idea used in  \cite{Ha74} is based on the notion of subordination (in the sense of Bochner, i.e.\ a random time change by an independent increasing L\'evy process), and on knowledge of the Hausdorff dimension of the range of a $\gamma$-stable subordinator $T_t^{\gamma}$ (cf.\ Lemma~\ref{ha} below).

The proof presented in \cite{Ha74} heavily relies on the fact that $X$ is a L\'evy process; a key ingredient is a criterion for the polarity of points in terms of the characteristic exponent of the L\'evy process $X$.  For general Markov processes such a result is not available, and so we need an essentially different approach.
The first problem which we encounter in the investigation of the level set \eqref{intro-e05}, is how to check that the process $X$ a.s.\ enters  $D$; in other words: when is the starting point $x$ \emph{regular} for $D$. We can overcome this problem using some abstract potential theory and the Kato class; this requires, however, upper and lower estimates for the transition density $p_t(x,y)$ of $X$ which allows us to characterize the notion of a Kato class (with respect to $p_t(x,y)$; see Definition~\ref{Kato}) and regular points for $D$. For $d$-sets this problem simplifies and, at least for certain values of $d$, any point in the topological boundary $\partial D$ is regular for $D$. Using the structure of the estimates for $p_t(x,y)$, we can establish similar assertions on the polarity of sets and regularity of points for the subordinate (i.e.\ time changed) process $X_{T_t^{\gamma}}$.

In Theorem~\ref{H-bounds} we use the indices $\gamma_{\inf}$ and $\gamma_{\sup}$---these characterize the set $D$ ``in the eyes'' of the time-changed process $X_{T_t^{\gamma}}$---to obtain uniform upper and lower bounds on the random set $\dim \{s\,:\, X_s(\omega)\in D\}$; here $D$ is a $d$-set and the process starts from a point $x$ which belongs to the topological closure $\overline{D}$ of $D$.  In the one-dimensional case we obtain (Proposition~\ref{cor-H}) the exact value of the Hausdorff dimension of the zero-level set $\{s:\, X_s(\omega)=0\}$.  This result can be pushed a bit further: in dimension one we show (Proposition~\ref{prop-H2}) that this value is also the Hausdorff dimension of the set of times, at which two independent copies of $X$ meet.

The second half of the paper is on collision sets. Motivated by our findings in Proposition~\ref{prop-H2} and the results from \cite{Ta66} and \cite{JP69}, we investigate  the Hausdorff dimension of the collision set
\begin{equation*}
    A(\omega):=
    \left\{x\in \real\,:\, \, X^1_t(\omega)=X_t^2(\omega)=x\quad \text{for some $t>0$}\right\}
\end{equation*}
of two independent copies $X^1$ and $X^2$ of $X$; from now on we assume that $X$ is one-dimensional and recurrent. Since recurrence reflects the behaviour of the process as time tends to infinity, it cannot be deduced from Assumption~A (which is essentially a condition on short times). Some examples of recurrent processes which fit our setting  are given in  Section~\ref{rec}. In order to get bounds on the Hausdorff dimension of $A(\omega)$, we compare the polar sets of the process  $(X^1, X^2)$ with the polar sets of symmetric stable processes with  parameters  $\alpha$ and $\beta$. The idea to use the range of a stable process as a ``gauge'' in order to express the Hausdorff dimension of a Borel set in $\rn$ is due to Taylor \cite{Ta66}; in its original version it heavily relies on the fact that the process $X$  is a L\'evy process. In the present paper, we use the symmetric stable (``gauge'') processes in a different way, especially when establishing the lower bound for the Hausdorff dimension.
%At this point the structure of the upper bound of the transition density estimates for $X$ and Frostman's lemma play the key role.

Let us briefly mention some known results. We refer to \cite{Xi04} for an extensive survey on sample path properties of L\'evy processes, in particular, for various  dimension results on  level, intersection and image sets. Most results essentially depend on the independence and stationarity of increments of L\'evy processes, while for general Markov processes much less is known. For L\'evy-type  processes the behaviour of the symbol of the corresponding generator allows us to  get the results on the Hausdorff dimension of the image sets, see e.g.\ \cite{Sch98}, \cite{KSW14}, and the monograph \cite{BSW14}; in \cite{Sh95} conditions are given, such that Markov processes collide with positive probability, and \cite{SX10} studies the Hausdorff and packing dimensions of the image sets of self-similar processes.

Our paper is organized as follows. In Section~\ref{set} we explain the notation and state our main results. Section~\ref{pot} is devoted to some facts and auxiliary statements from probabilistic potential theory; these are interesting in their own right. The proofs of the main results are given in Sections~\ref{s4} and \ref{s5}. Examples of recurrent processes, which satisfy Assumption~A can be found in Section~\ref{rec}. Finally, the (rather technical) proofs of some auxiliary statements are given in the appendix.

\section{Setting and main results}\label{set}

We begin with the description of the class of stochastic processes which we are going to consider. Denote by $C_\infty^k(\rn)$ and $C_c^k(\rn)$  the spaces of $k$ times continuously differentiable functions which vanish at infinity (with all derivatives) and which are compactly supported, respectively. For $f\in C_\infty^2(\rn)$ we consider the following L\'evy-type operator
\begin{equation}\label{set-e05}
    \mathcal{L}f(x)
    :=a(x) \cdot \nabla f(x) + \int_\rno \big(f(x+h)-f(x)-h \cdot \nabla f(x) \I_{(0,1)}(|h|)\big)m(x,h)\,\mu(dh),
\end{equation}
where $a:\rn\to\rn$, $m:\rn\times\rn\to (0,\infty)$ are measurable functions and $\mu$ is a L\'evy measure, i.e.\ a measure on $\rno$ such that $\int_\rno \big(1\wedge |h|^2\big)\mu(dh)<\infty$.

Denote by $\hat f(x) := (2\pi)^{-n}\int_\rn f(x)e^{-ix\cdot\xi}\,dx$ the Fourier transform. It is not hard to see that we can rewrite $\mathcal{L}$ as a pseudo-differential operator
\begin{equation*}%\label{set-e10}
    \mathcal{L}f(x)
    := - \int_\rn  e^{i\xi \cdot x} q(x,\xi) \hat f(\xi)\,d\xi,
    \qquad f\in C_c^\infty(\rn),
\end{equation*}
with \emph{symbol} $q:\rn\times\rn\to\comp$. The symbol is given by the L\'evy--Khintchine representation
\begin{equation}\label{set-e15}
    q(x,\xi)
    =
    -ia(x)\cdot\xi + \int_\rno \big(1-e^{ih\cdot\xi}+ih\cdot\xi\I_{(0,1)}(|h|)\big) m(x,h)\,\mu(dh).
\end{equation}

We will frequently compare the variable-coefficient operator $ \mathcal{L}$ with an operator $ \mathcal{L}_0$ (with bounded coefficients), defined by
$$
  \mathcal{L}_0 f(x)
    =  - \int_\rn e^{i x \cdot \xi} q(\xi) \hat f(\xi)\,d\xi,
$$
with the real-valued symbol
\begin{equation}\label{set-e20}
    q(\xi) = \int_\rno  \big(1-\cos (\xi \cdot h)\big)\, \mu(dh).
\end{equation}
The symbol $q(\xi)$ is the characteristic exponent of a symmetric L\'evy process $Z_t$ in $\rn$, i.e.\ $\Ee e^{i \xi\cdot Z_t} = e^{-tq(\xi)}$.
Define
\begin{equation*}
    q^U(\xi)
    := \int_\rno \big((\xi \cdot h)^2\wedge 1\big)\,\mu(dh)
    \quad\text{and}\quad
    q^L(\xi)
    := \int_{0<|\xi\cdot h|\leq 1} (\xi \cdot h)^2\,\mu(dh)
\end{equation*}
and
\begin{equation*}
    q^*(r)
    := \sup_{\ell\in \mathds{S}^n}  q^U( r\ell),
\end{equation*}
where $\mathds{S}^n$ is the  unit sphere in $\rn$. The functions $q^U$ and $q^L$ are, up to multiplicative constants, upper and lower bounds for $q(\xi)$ (cf.~\cite{KK12a,K13}):
\begin{equation*}%\label{set-e25}
    (1-\cos 1) q^L (\xi)
    \leq q(\xi)
    \leq 2 q^U (\xi).
\end{equation*}
 The key regularity assumption in \cite{KK13a,KK13b} is the following comparison result:
\begin{equation}\label{set-e30}
    \exists \kappa \geq 1 \quad
    \forall r\geq 1\::\:
    q^*(r) \leq  \kappa \inf_{\ell \in \mathds{S}^n} q^L( r \ell).
\end{equation}
This condition means that the function $q(\xi)$  does not oscillate ``too much''. For example, if $q(\xi)=|\xi|^\alpha$ one can check that \eqref{set-e30} holds true with $\kappa=2/\alpha$. Motivated by this example, we use the notation
\begin{equation}\label{set-e31}
    \alpha:= 2/\kappa%  \quad\text{with $\kappa\geq 1$ from \eqref{set-e30}.}
\end{equation}
with $\kappa\geq 1$ from \eqref{set-e30}. Moreover, \eqref{set-e30} implies, see \cite{KK12a,K13}, that
\begin{equation}\label{set-e33}
    q(\xi)\geq c |\xi|^\alpha, \quad |\xi|\geq 1.
\end{equation}
We refer to \cite{KK12a} for examples which illustrate this condition.

In  \cite{KK13b} it was shown that, under the following assumptions
\begin{assumption-A}\mbox{}
\begin{itemize}
\item[\bfseries 1)]
    The L\'evy measure $\mu$ is such that \eqref{set-e30} holds;
\item[\bfseries 2)]
    There exist constants $c_1, c_2, c_3>0$, such that $|a(x)|\leq c_1$ and $c_2\leq m(x,u)\leq c_3$;
\item[\bfseries 3)]
    The functions $a(x)$ and $ m(x,u)$ are locally  H\"older continuous in $x$ with some index $\lambda\in (0,1]$;
\item[\bfseries 4)]
    Either $\alpha>1$, with $\alpha$ as in \eqref{set-e30}, \eqref{set-e31}, or $a(x)\equiv 0$ and $m(x,h)=m(x,-h)$, $\mu(dh)=\mu(-dh)$,
\end{itemize}
\end{assumption-A}

\noindent
the operator $\mathcal{L}$ extends to the generator  of a (strong) Feller process $X$, which has a transition probability density $p_t(x,y)$. This density is continuous for $(t,x,y)\in [t_0,\infty)\times \rn\times \rn$, $t_0>0$, and satisfies the following upper and lower bounds:
\begin{equation}\label{set-e35}
    p_t(x,y) \geq \rho_t^n  f_{\textrm{low}}(x \rho_t),\quad t\in (0,1],\; x,y\in \rn,
\end{equation}
and
\begin{equation}\label{set-e36}
    p_t(x,y)\leq   \rho_t^n \big(f_{\textrm{up}} (\rho_t \:\cdot) * Q_t\big)(y-x),\quad t\in (0,1],\; x,y\in \rn,
\end{equation}
where $(Q_t)_{t\geq 0}$ is  a family of sub-probability measures,
\begin{gather*}
    \rho_t :=\inf\{ r>0: \,\, q^*(r)\geq 1/t\}, \\
    f_{\textrm{low}}(z):= a_1 (1-a_2|z|)_+
    \quad\text{and}\quad
    f_{\textrm{up}}(z):=a_3 e^{-a_4 |z|},\quad z\in \real;
\end{gather*}
($a_i>0$, $i=1,\dots,4$, are constants and $x_+:= \max(x,0)$).  The family of sub-probability measures $(Q_t)_{t\geq 0}$ is explicitly constructed in \cite{KK13b}; for our purposes the exact form of the $Q_t$ is not important.

\medskip\noindent
\boldmath
\textbf{Unless otherwise specified, $X=(X_t)_{t\geq 0}$ will always denote an $\rn$-valued Feller process as above, with law  $\Pp^x(X_t\in dy) = p_t(x,y)\,dy$, $t>0$.}
\unboldmath

\medskip

There are many Feller and L\'evy-type processes satisfying the conditions required in Assumption~A. Note that the integro-differential structure of the generator---as in \eqref{set-e15}, but with a jump kernel (compensator of the jumping measure) $N(x,dh)$ instead of $m(x,h)\,\mu(dy)$ and with a second-order term---is, in fact, necessary for Feller processes and more general semimartingales, at least if the test functions $C_c^\infty(\rn)$ are in the domain of the generator, see \cite{BSW14}. This means that the main restriction is the fact that $N(x,dy)$ is absolutely continuous w.r.t.\ some L\'evy measure and the absence of a second-order diffusion part; just as in the L\'evy case, the latter would dominate the short-time path behaviour. Below we give a few typical examples of Feller processes satisfying our assumptions.
\begin{itemize}
\item  Any Rotationally symmetric L\'evy process whose L\'evy measure has a (rotationally symmetric) density $g(|u|)$ satisfying\footnote{We write $f(t)\asymp g(t)$ or $f\asymp g$ if there is an absolute constant $0<c<\infty$ such that $c^{-1}f(t)\leq g(t)\leq cf(t)$ for all $t$ (in the specified domain)}
\begin{equation}\label{set-e38}
     \int_0^a  r^2 g(r)\,dr\asymp  a^2 \int_a^\infty g(r)\,dr.
\end{equation}
 A concrete example when such condition is satisfied is given in Example~\ref{exa1} below.

\item  Any L\'evy process whose  L\'evy measure is radially symmetric, i.e.
$$
   \mu(dh)= \int_0^\infty \int_{\mathds{S}^n} \delta_{r\zeta}(dh)\,m(dr)\,\mu_0(d\zeta),
$$
where $\mu_0$ is a finite measure on $\mathds{S}^n$; we assume, in addition, that $r\mapsto m(\real\setminus (-r^2,r^2))$  is regularly varying at $0$. Under these assumptions $q^L(\xi)\asymp f(|\xi|)$, where $f(|\xi|) = \int_{r|\xi|\leq 1} r^2 |\xi|^2\, m(dr)$ is regularly varying  at infinity as we have the representation
$$
    f(|\xi|)=|\xi|^2 \int_0^{1/|\xi|^2} m\{ r\,:\, r^2>s\}\,ds,
$$
see \cite[Proposition 1.5.8]{BGT}. Fix some $\ell\in \mathds{S}^n$, and rewrite $q^L$ as
\begin{equation}\label{set-e39}
\begin{split}
    q^L(|\xi|\ell)
    &= \int_{|\xi||\ell\cdot h|\leq 1} |\xi|^2 (\ell\cdot h)^2\,\mu(dh)\\
    &= \int_{h\neq 0} \I_{\{|\xi||\ell\cdot h|\leq 1\}} |\xi|^2 (\ell\cdot h)^2\,\mu(dh)\\
    %&= \int_{h\neq 0} \int_0^\infty \int_{\mathds{S}^n} \I_{\{|\xi||\ell\cdot h|\leq 1\}} |\xi|^2 (\ell\cdot h)^2\,\delta_{r\zeta}(dh)\,m(dr)\,\mu_0(d\zeta)\\
    &= \int_0^\infty \int_{\mathds{S}^n} \I_{\{|\xi| r |\ell\cdot \zeta|\leq 1\}} |\xi|^2 r^2 (\ell\cdot \zeta)^2\,m(dr)\,\mu_0(d\zeta).
    \end{split}
    \end{equation}
    Since $|\ell\cdot\zeta|\leq 1$,  we get $\I_{\{|\xi| r |\ell\cdot \zeta|\leq 1\}}\geq \I_{\{|\xi| r \leq 1\}}$, and so
    \begin{align*}
    q^L(|\xi|\ell) &\geq %\int_0^\infty \int_{\mathds{S}^n} \I_{\{|\xi| r \leq 1\}} |\xi|^2 r^2 (\ell\cdot \zeta)^2\,\,m(dr)\,\mu_0(d\zeta)\\
   % &=
    \int_0^\infty \I_{\{|\xi| r \leq 1\}} |\xi|^2 r^2 \,m(dr) \int_{\mathds{S}^n}(\ell\cdot \zeta)^2\,\mu_0(d\zeta)\\
    &= f(|\xi|) \int_{\mathds{S}^n}(\ell\cdot \zeta)^2\,\mu_0(d\zeta).
\end{align*}
 On the other hand, the last line of \eqref{set-e39} reads
\begin{align*}
    q^L(|\xi|\ell)
    &= \int_{\mathds{S}^n} f(|\xi| |\ell\cdot \zeta|) \,\mu_0(d\zeta).
\end{align*}
As $f(r)$ is regularly varying  at infinity, there exists some $C>0$, such that  $f(cr)\leq C f(r)$ for any $c\in (0,1]$ and sufficiently large values of $r\gg 1$, see \cite[Theorem 1.5.6]{BGT}.   Therefore, we get
$$
    q^L(|\xi|\ell)\leq C \mu_0(\mathds{S}^n)f(|\xi|),\quad |\xi|\gg 1.
$$
Observe also, that the function $q^U$ is differentiable almost everywhere, and  the derivative with respect to the radial component equals
\begin{equation*}
    \frac{\partial}{\partial r} q^U(r \ell ) = \frac{2}{r} q^L(r\ell).
\end{equation*}
for any $\ell \in \mathds{S}^n$ and $r>0$. Therefore, we deduce with our previous calculations that
$$
    1\asymp
    \lim_{r\to \infty } \frac{q^L(\lambda r\ell)}{q^L(r\ell)}
    \asymp
    \lim_{r\to \infty } \frac{q^U(\lambda r\ell)}{q^U(r\ell)};
$$
for the second equivalence relation we use l'Hospital's rule. Thus, condition \eqref{set-e30} holds true.

\item Any L\'evy process from the previous example, which is  perturbed with a non-constant drift $a(x)$ and such that $q(\xi) \geq c\,|\xi|^{1+\epsilon}$ for some $\epsilon>0$;
\item (Weak) solutions  to SDEs driven by symmetric $\alpha$-stable L\'evy noise $(1<\alpha <2)$ and H\"older continuous coefficients, see \cite{KK14a} for the existence of such weak solutions, as well as for a simplified version of the parametrix method.
\item Stable-type processes (in the sense of Z.-Q.\ Chen and T.\ Kumagai) where $m(x,h)$ is jointly continuous, bounded and bounded away from $0$ and $\mu(dh) = |h|^{-\alpha-d}\,dh$.
\end{itemize}
In general, the main  problem  is to show that \eqref{set-e30} holds true, which is a condition on the L\'evy measure. To wit, this condition holds true for the ``discretized version'' of an $\alpha$-stable L\'evy measure in $\rn$:
$$
    \mu(dh)=\sum_{k=-\infty}^\infty 2^{k\gamma}m_{k,v}(dh), \quad 0<\gamma<2\upsilon,
$$
where  $m_{k,\upsilon}(dh)$ is  the uniform distribution  on the sphere $\mathbb{S}_{k,\upsilon}$ centered at $0$ with radius $2^{-k\upsilon}$, $\upsilon>0$, $k\in \mathbb{Z}$,   $\upsilon>0$, $0<\gamma<2\upsilon$, see \cite{K13}. In this example $q^U(\xi)\asymp q(\xi)\asymp q^L(\xi)\asymp |\xi|^\alpha$, where  $\alpha=\gamma/\upsilon \in (0,2)$; see \cite{KK12a} for further examples in this direction. On the other hand, L\'evy measures of the form $\sum_{k=0}^\infty a_k\delta_{h_k}$, $a_k, h_k >0$ for rapidly growing weights $a_k\to\infty$ and $h_k\to 0$ are exactly those measures which create oscillations in the symbol $q$, making \eqref{set-e30} impossible, see \cite[Example 1.1.15]{fjs-2}.

In Section~\ref{rec} we consider further examples of processes which satisfy Assumption~A and are recurrent (which is needed in the second main result of our paper).

\normal

\medskip
In order to  state our  result on the bound for the Hausdorff dimension of level sets we need to define two auxiliary indices.
Recall that a set $D$ is called a $d$-set, if  there exists a measure $\varpi\in \mathcal{M}_b^+(\overline{D})$, $\supp\varpi = \overline{D}$,  such that
\begin{equation}\label{set-e45}
    c_1 r^d \leq \varpi\big(B(x,r)\cap \overline{D} \big) \leq c_2 r^d, \quad  x\in \overline{D},\; r>0;
\end{equation}
the corresponding measure $\varpi$ is called a $d$-measure.
Denote by  $\mathcal{M}_b^+(\overline{D})$ the family of all finite Borel measures with support in  $\overline{D}\subset \rn$. For a $d$-set $D$ we define
\begin{align}
\label{set-e46}
    \gamma_{\inf}
    &:= \inf\left\{ \gamma\in [0,1]\,:\,\int_0^1  \frac{\varpi\big(B(x,r)\big)}{(q^*)^{\gamma}(1/r)}\,\frac{dr}{r^{n+1}} < \infty, \,\,\text{for a $d$-measure $\varpi$ on $D$}\right\}\\
    &=\nonumber
    \inf\left\{\gamma\in [0,1]\,:\,\int_0^1  \frac{r^d}{(q^*)^{\gamma}(1/r)}\,\frac{dr}{r^{n+1}} < \infty\right\},\\
\label{set-e47}
    \gamma_{\sup}
    &:= \sup\left\{\gamma \in [0,1]\,:\, x\mapsto \int_0^1 \frac{\varpi\big(B(x,r)\cap \overline{D}\big)}{(q^*)^{\gamma}(1/r)}\,\frac{dr}{r^{n+1}} \;\;\text{is unbounded}\;\;\forall \varpi\in \mathcal{M}_b^+(\overline{D})\right\}.
\end{align}

Let us give an intuitive explanation of the meaning of the indices $\gamma_{\inf}$ and $\gamma_{\sup}$.
Denote by $T^{\gamma} = (T_t^{\gamma})_{t\geq 0}$, $\gamma\in (0,1)$, a $\gamma$-stable subordinator, i.e.\ a real-valued L\'evy process with increasing sample paths such that $t^{-1/\gamma} T^{\gamma}_t=T^{\gamma}_1$ in distribution for all $t>0$. Assume that $T^{\gamma}$ is independent of $X$. Intuitively, $\gamma_{\inf}$ is the smallest $\gamma$ for which the time-changed process $X_{T^\gamma_t}$ still can see the set $\overline{D}$,  and $\gamma_{\sup}$ is the largest $\gamma$, for which $\overline{D}$ is polar for $X_{T^\gamma_t}$.

We can now state our first main result.
\begin{theorem}\label{H-bounds}
    Suppose that the Feller process $X$ with generator $\mathcal{L}$ satisfies  Assumption~A, and $D=\overline D\subset \rn$ is a closed $d$-set with $d>n-\alpha$. If $x\in D$, then\footnote{Here, as well as in the rest of the paper, ``$\dim$'' stands for the Hausdorff dimension.}
    \begin{equation}\label{set-e50}
        1-\gamma_{\inf}
        \leq \dim\{s\,:\, X_s^x \in D\}
        \leq 1-\gamma_{\sup}, \quad \Pp^x\text{-a.s.}
    \end{equation}
    where $\gamma_{\inf}$ and $\gamma_{\sup}$ are given by \eqref{set-e46} and \eqref{set-e47}, respectively.
\end{theorem}

In the one-dimensional case we can get a result which closely resembles those in \cite{Ha74} for L\'evy processes.  Denote by
$$
    X^{-1}(\{0\},\omega):= \{ s>0\,:\, X_s(\omega)=0\}, \quad \text{where $X_0(\omega)=0$},
$$
the zero-level set of $X$ and set
\begin{equation*}
    \gamma^* := \inf\left\{ \gamma\in [0,1]\,:\,\int_0^1 \frac{1}{(q^*(1/s))^\gamma}\,\frac{ds}{s^2} <\infty \right\}.
\end{equation*}
The corollary below follows from Theorem~\ref{H-bounds} if we take $D=\{0\}$,  $d=0$ and $\alpha>1$; in this case points are non-polar for $X$.
\begin{corollary}\label{cor-H}
    Let $X$ be a Feller process with generator $\mathcal{L}$ and suppose that Assumption~A is satisfied. Let $n=1$ and $\alpha>1$.  Then
    \begin{equation*}
        \dim X^{-1}(\{0\},\omega) = 1-\gamma^* \quad \Pp^0\text{-a.s.}
    \end{equation*}
    In particular, if $q^*(\xi)\asymp |\xi|^\alpha$ ($|\xi|\geq 1$), then $\gamma^*= 1/\alpha$.
\end{corollary}

Corollary~\ref{cor-H} can also be used to calculate the Hausdorff dimension of the set of collision times of independent copies $X^1$, $X^2$ of $X$:
\begin{equation*}
    \Theta(\omega)
    := \left\{ t\geq 0\,:\, X^1_t =X^2_t=x\quad \text{for some $x\in \rn$}\right\}.
\end{equation*}
\begin{proposition}\label{prop-H2}
    Suppose that $X$ is a one-dimensional $(n=1)$ Feller process with generator $\mathcal{L}$ and that Assumption~A is satisfied. Let $\alpha>1$, and denote by $X^{1}$ and $X^{2}$ two independent copies of $X$. Then
    \begin{equation*}
        \dim\Theta(\omega)= 1-\gamma^*  \quad \Pp\text{-a.s.}
    \end{equation*}
\end{proposition}

Our second main result concerns the Hausdorff dimension of the collision set
\begin{align}\label{set-e65}
    A(\omega)
    &:= \left\{x\in \real\,:\, X^1_t(\omega) = X_t^2(\omega) = x \quad \text{for some $ t>0$}\right\}\\
    \notag&\qquad (X^1, X^2\text{\ are two independent copies of $X$}).
\end{align}

\begin{theorem}\label{t-col}
    Let $X$ be a one-dimensional $(n=1)$ Feller process with generator $\mathcal{L}$ and suppose that Assumption~A is holds. If $X$ is recurrent and if the function $q(\xi)$ from \eqref{set-e20} satisfies
    \begin{equation}\label{set-e70}
        c_1|\xi|^\alpha \leq  q(\xi)\leq c_2 |\xi|^\beta\quad \text{for all $|\xi|\geq 1$},
    \end{equation}
    for some constants $c_1,c_2>0$ and $1<\alpha\leq \beta<2$, then the Hausdorff dimension of the collision set $A(\omega)$ is estimated from above and below as
    \begin{equation*}
        \alpha-1
        \leq  \dim  A(\omega)
        \leq \beta-1 \quad \Pp^x\text{-a.s.\ for all $x\in\real$.}
    \end{equation*}
\end{theorem}

\section{Some auxiliary results from potential theory}\label{pot}

A central problem is which points can be hit by the process $X$. For this we need a few tools from potential theory. The following definition is taken from \cite{KT07}.

\begin{definition}\label{Kato}
    Let $(X_t)_{t\geq 0}$ be an $\rn$-valued Markov process admitting a transition density $p_t(x,y)$ and $\varpi$ a Borel measure on $\rn$. The measure $\varpi$ belongs to the \emph{Kato class} $\Kato$ with respect to $p_t(x,y)$,  if
    \begin{equation}\label{pot-e05}
        \lim_{t\to 0} \sup_{x\in \rn} \int_0^t \int_\rn p_s(x,y)\,\varpi(dy)\,ds = 0.
    \end{equation}
\end{definition}

Let $r_\lambda(x,y)$, $\lambda>0$,  be the $\lambda$-potential density of $X$, i.e.
$$
    r_\lambda(x,y) := \int_0^\infty e^{-\lambda s} p_s(x,y)\, ds.
$$
We can extend the resolvent operator from functions $f\in L_1(\rn)$ to (finite) measures: For $\lambda>0$ and any  finite  measure $\varpi$ we can define the operator
\begin{equation*}
    R_\lambda \varpi(x)
    := \int_0^\infty \int_\rn e^{-\lambda s} p_s(x,y)\,\varpi(dy)\,ds
    = \int_\rn r_\lambda(x,y)\,\varpi(dy).
\end{equation*}
A Borel set $D\subset\rn$ is \emph{polar} for $X=(X_t)_{t\geq 0}$, if $\Pp^x(\tau_D<\infty)=0$ for all $x\in \rn$, where
\begin{equation*}
    \tau_D := \inf\{t>0\,:\, X_t \in D\}
\end{equation*}
is the first hitting time of the set $D$.

\begin{remark}\label{equiv}
It is shown in \cite{KT06} that \eqref{pot-e05}
%and \eqref{pot-e06}
is equivalent to ``$\lim_{\lambda\to \infty}  \sup_x R_\lambda \varpi(x)=0$''.
%and ``$\exists\lambda>0\::\: \sup_x R_\lambda \varpi(x)<\infty$'', respectively.
The set $D$ is polar if and only if $R_0\varpi(x)$ is unbounded for any %NOTE: any really means all!
finite non-zero measure $\varpi$ with compact support contained in $\overline{D}$, see \cite[p.~285]{BG68}.
\end{remark}

In order to make sure that the process $X$ enters the set $D$, we need to take the starting point $x$ from the \emph{fine closure} (i.e.\ the closure in the fine topology) of $D$. Recall from \cite[p.~87, Exercise 4.9]{BG68} that the fine closure $\widetilde D$ of a set $D$ is $D\cup D^r$, where $D^r$ denotes the set of \emph{regular points} of $D$, i.e.
$$
    D^r := \left\{x\in \rn\,:\,\Pp^x(\tau_D=0)=1\right\}.
$$
We need to characterize the regular points for $D$. The following elementary result \emph{should} be known, but we could not find a reference and so we include the short proof.
\begin{lemma}\label{lem-reg}
    Let $D\subset \rn$ and assume that there exists a finite measure $\varpi\in\Kato$ (w.r.t.\ $p_t(x,y)$) with $\supp\varpi=\overline{D}$. If a point $x\in\rn$ satisfies
    \begin{equation}\label{pot-e15}
        \liminf_{\lambda\to \infty} \frac{R_\lambda \varpi(x)}{\sup_{y\in \overline{D}} R_\lambda \varpi(y)}
        = c(x) > 0,
    \end{equation}
    then $x$ is regular for $D$. In particular, if a point $x$ is not regular for $D$, then the constant $c(x)$ in \eqref{pot-e15} is necessarily equal to $0$.
\end{lemma}
\begin{proof}
    Let $\varpi$ be a finite measure such that $\supp\varpi= \overline{D}$ and $\varpi\in \Kato$. By \cite[Vol.~1, p.~194, Theorem~6.6]{Dy65}, there exists a continuous additive functional\footnote{that is, $A_{t+s}=A_s + A_t\circ \theta_s$ for any $t,s>0$ where $\theta_s$ is the shift operator.} $A_t$  satisfying
    $$
         \Ee^x A_t = \int_0^t \int_\rn p_s(x,y)\,\varpi(dy)\,ds.
    $$
    Using standard arguments, we find for any $\lambda>0$ and $x\in \rn$
\begin{align}\label{pot-e20}
    \Ee^x  \int_0^m e^{-\lambda t}\, dA_t
    %&=  e^{-\lambda t} \Ee^x A_t\Big|_0^m - \Ee^x \Big[\int_0^m A_t de^{-\lambda t}\Big] \\
    %&= e^{-\lambda m} \Ee^x A_m - \int_0^m \Ee^x A_t d e^{-\lambda t} \\
    %&= e^{-\lambda m} \Ee^x A_m - e^{-\lambda t} \Ee^x A_t \Big|_0^m +
    = \int_0^m e^{-\lambda t}\, d\Ee^x A_t%\\%&
    = \int_0^m e^{-\lambda t} p_t(x,y)\,\varpi(dy)\,dt.
\end{align}
    Passing to the limit as $m\to \infty$, we get
\begin{equation}\label{pot-e25}
    \Ee^x  \int_0^\infty e^{-\lambda t} \,dA_t
    = R_\lambda \varpi(x).
\end{equation}
    Let $\tau:=\tau_D$ be the hitting time of the set $D$. By construction, the additive functional $A_t$ satisfies $A_t=0$ for $t<\tau$. Thus,
\begin{align*}
    \lambda\Ee^x &\int_0^m e^{-\lambda t} A_t \,dt\\
    &= \lambda \underbrace{\Ee^x \int_0^{ m} e^{-\lambda t} A_t \I_{\{\tau > m\}}\,dt}_{=0}
       + \lambda \Ee^x \int_{\tau}^m e^{-\lambda t} A_t \I_{\{\tau\leq m\}}\,dt\\
    &= \lambda \Ee^x \int_0^{m-\tau}e^{-\lambda (t+\tau)}  A_{t+\tau} \I_{\{\tau\leq m\}} \,dt\\
   %&=\lambda \Ee^x  \Big[e^{-\lambda \tau}\I_{\{\tau<m\}} \Ee^x \Big[ \int_0^{m-\tau} e^{-\lambda t} A_{t+\tau} \,dt \Big| \mathcal{F}_{\tau}\Big] \Big]\\
    &= \Ee^x  \left[e^{-\lambda \tau} \I_{\{\tau\leq m\}} \Ee^{X_{\tau}} \left( \lambda \int_0^{m-\tau} e^{-\lambda t} A_t\,dt \right) \right]\\%+ \lambda \Ee^x e^{-\lambda \tau} A_{\tau}\\
    &=\Ee^x\left[  e^{-\lambda \tau}\I_{\{\tau \leq m\}}  \Ee^{X_{\tau}} \left(\int_0^{m-\tau}  e^{-\lambda  t} \,dA_t\right)\right] - e^{-\lambda  m} \,\Ee^x \left[  \I_{\{\tau \leq m\}} \Ee^{X_{\tau}} A_{m-\tau}\right]\\
    &= \Ee^x\left[  e^{-\lambda \tau}\I_{\{\tau \leq m\}}  \Ee^{X_{\tau}} \left( \int_0^{m-\tau}  e^{-\lambda  t} \,dA_t\right)\right] - e^{-\lambda m} \Ee^x A_m.
\end{align*}
    For the last step we used the continuity of $A_t$ to get $A_\tau =0$ and, by the additive property,
\begin{align*}
    \Ee^x  A_m
    =\Ee^x A_{\tau} + \Ee^x \left[A_{m-\tau}\circ \theta_{\tau}\I_{\{\tau \leq m\}}\right]
    &= \Ee^x \left[ \Ee^x \left( A_{m-\tau}\circ  \theta_{\tau} \:\middle|\: \mathcal{F}_{\tau}\right) \right]\\
    &= \Ee^x \left[\I_{\{\tau \leq m\}}\Ee^{X_{\tau}} A_{m-\tau}\right].
\end{align*}
    These calculations, when combined with \eqref{pot-e20} and integration by parts, yield
\begin{equation*}
    \Ee^x \int_0^m e^{-\lambda t}\,dA_t
    = \Ee^x\left[  e^{-\lambda \tau}\I_{\{\tau \leq m\}}  \Ee^{X_{\tau}} \left( \int_0^{m-\tau}  e^{-\lambda  t} \,dA_t\right)\right],
\end{equation*}
    and passing to the limit as $m\to \infty$ we finally arrive at
\begin{equation*}
    R_\lambda \varpi(x) = \Ee^x \left[ e^{-\lambda \tau} R_\lambda \varpi(X_{\tau})\right].
\end{equation*}
    Since $X_{\tau}\in \overline{D}$, the last equality implies
\begin{equation}\label{pot-e40}
    \frac{R_\lambda \varpi(x)}{\sup_{y\in \overline{D}} R_\lambda\varpi(y)}\leq  \Ee^x e^{-\lambda \tau}.
\end{equation}
    Note that $ \{\tau>0\}$ is a ``tail event'', i.e.\ it has probability $0$ or $1$. Taking the lower limit $\liminf_{\lambda \to \infty}$ on both sides, we get a contradiction to \eqref{pot-e15}, unless $\tau\equiv 0$. Thus, $\Pp^x(\tau>0)=0$.
\end{proof}

\begin{remark}
    For a symmetric Markov process $X$, the relation \eqref{pot-e25} is known for all   measures which have \emph{finite energy integrals}, see \cite[pp.~223--226, Theorem~5.1.1, Lemma~5.1.3]{FOT11}.
\end{remark}

It is possible to give a more explicit sufficient condition for a point $x$ to be regular for $D$; this requires further knowledge of the structure of $D$, for instance that $D$ is a $d$-set.
\begin{lemma}\label{lem-dmeas}
    Let $D\subset \rn$  be a $d$-set and assume that the corresponding $d$-measure $\varpi$ belongs to $\Kato$ w.r.t.\ $p_t(x,y)$. Then any point of $\overline{D}$  is regular for $D$, i.e.\ $\overline{D} = D \cup D^r = \widetilde D$.
\end{lemma}
In order to keep the presentation transparent, we defer the rather technical proof of this lemma to the appendix.

Here is a criterion for the non-polarity of a set $D$ based on the inequality \eqref{pot-e40}.
\begin{corollary}\label{poten}
    Assume that there exists some $\varpi\in \Kato$  w.r.t.\ $p_t(x,y)$ such that $\supp\varpi= \overline{D}$.  Then the set  $D$  is non-polar for $X$, i.e.
    \begin{equation}\label{pot-e50}
        \Pp^x(\tau_D<\infty) > 0.
    \end{equation}
\end{corollary}
\begin{proof}
    We know from \cite{KT06}, see also Remark~\ref{equiv}, that $\varpi\in \Kato$ satisfies $\sup_x R_\lambda \varpi(x)<\infty$ for some $\lambda>0$. From \eqref{pot-e40} we derive
    \begin{equation*}
        \frac{R_\lambda  \varpi(x)}{\sup_{y\in \overline{D}} R_\lambda\varpi(y)}\leq  \Pp^x (\tau_D<\infty).
    \end{equation*}
    Let us show that $R_{\lambda}\varpi(x)>0$.  For this we show
that
$$
    p_t(x,y)>0\quad \text{for all $ t>0$,\,\, $x,y\in \rn$.}
$$
There is a minimal $N$, such that the distance from $x$ to $y$ can be covered by $N$ balls of radius less than $(2a_2\rho_{t/N})^{-1}$ (where $a_2>0$ is the constant appearing in the representation of $f_{low}$), i.e.\ the smallest $N$, for which the inequality
\begin{equation}\label{pot-e52}
    \frac{|x-y|}{N}\leq \frac{1}{a_2\rho_{t/N}}
\end{equation}
holds true. Observe that $q^*(r)\leq c_1 r^2$, $r\geq 1$, implying $c_2 t^{-1/2}\leq \rho_t$, for all $t$ small enough.  Hence, \eqref{pot-e52} is valid for all $N \geq (a_2 c_2 |x-y|)^2/{t}$.
Therefore, putting $y_0=x$, $y_N=y$, we get
\begin{align*}
p_t(x,y)&= \int_\rn \dots \int_\rn \left(\prod_{i=1}^N p_{t/N}(y_{i-1},y_i) \right) dy_1\dots dy_N\\
&\geq \int_{B(y_0,(2a_2\rho_{t/N})^{-1})} \dots \int_{B(y_{N-1}, (2a_2\rho_{t/N})^{-1})}\prod_{i=1}^N p_{t/N}(y_{i-1},y_i) \,dy_i \\
&\geq c_0 \rho_{t/N}^{Nn}.
\end{align*}
In the last line we use \eqref{set-e35} which gives
$$
    p_{t/N} (y_{i-1},y_i)\geq  2^{-1}a_1 \rho_{t/N}^n \quad
    \forall y_i\in B(y_{i-1},(2a_2\rho_{t/N})^{-1}).
$$
Thus, the transition probability density $p_t(x,y)$ is strictly positive, which implies
$$
    R_\lambda\varpi(x)
    \geq e^{-\lambda}\int_0^1\int_D p_t(x,y) \,\varpi(dy) \,dt
    > 0.
$$
Hence, we get \eqref{pot-e50}.
\end{proof}

\begin{remark}\label{low-rem}
\textbf{a)}
    Under the assumptions of Corollary~\ref{poten} one has $\Pp^x(\tau_D<\infty)> c_K$ uniformly for all $x\in K$ where $K\subset \rn$ is a compact set.

\medskip\textbf{b)}
    If, in addition, the process $X$ is recurrent, then $\Pp^x(\tau_D<\infty)=1$, see \cite[p.~60]{Sh88}.

\medskip\textbf{c)}
    Suppose that $X$ is one-dimensional ($n=1$) and $\int_1^\infty q^*(s)^{-1}\,ds<\infty$.
    Then there exists a local time for any point $x\in \rn$, see \cite{KK13b}. Let $D=\{x\}$, where $x$ is the starting point of $X_t$. Then $R_\lambda\varpi(x) = \sup_{y\in \real} R_\lambda\varpi(y)$, i.e.\ the left-hand side of \eqref{pot-e15} is equal to $1$, implying that every point is regular for itself.

    On the other hand, if $n\geq 2$, we always have $\int_{|\xi|\geq 1} q^*(\xi)^{-1}\,d\xi = \infty$, i.e.\ for $n\geq 2$ points are polar.
\end{remark}

\section{Proof of Theorem~\ref{H-bounds} and  Proposition  \ref{prop-H2}}\label{s4}

Throughout this section $X=(X_t)_{t\geq 0}$ is a Feller process as in Section~\ref{set}. Let $(\Omega^*, \mathcal{F}^*,\Pp^*)$ be a further probability space and define on this space a $\gamma$-stable subordinator $T^{\gamma} = (T_t^{\gamma})_{t\geq 0}$, $\gamma\in (0,1)$. $T_t^{\gamma}$ has a transition probability density $\sigma_t^{(\gamma)}(s)$, and
\begin{equation*}%\label{s4-e05}
    \int_0^\infty e^{-\lambda s}\sigma_t^{(\gamma)}(s)\,ds
    = e^{-t \lambda^\gamma}, \quad \lambda>0,\; t>0.
\end{equation*}
From this we immediately get the following scaling property
\begin{equation}\label{s4-e07}
    \sigma_t^{(\gamma)} (s)
    = t^{-1/\gamma} \sigma_1^{(\gamma)}(s t^{-1/\gamma}).
\end{equation}
Let $X^{\gamma}_t:= X_{T_t^{\gamma}}$ be the subordinate  process. Its transition probability density $p^{(\gamma)}_t(x,y)$ is given by
\begin{equation}\label{s4-e10}
    p^{(\gamma)}_t(x,y)
    = \int_0^\infty p_s(x,y) \sigma_t^{(\gamma)}(s)\,ds,
\end{equation}
see, for example, \cite[Theorem~4.3.1]{Ja01}.

The technical proof of the following lemma is deferred to the appendix. Recall that $\Kato$ denotes the Kato class of measures, cf.\ Definition~\ref{Kato}. If $\gamma=1$, $T_t^{(\gamma)}\equiv t$, and the `subordinate' kernel $p^{(1)}_t(x,y)$ is just $p_t(x,y)$.
\begin{lemma}\label{lem-SK}
\textbf{\upshape a)}
    Suppose that $\varpi$ satisfies
    \begin{equation}\label{s4-e15}
        \int_0^1 \sup_x \frac{\varpi(B(x,r))}{(q^*)^{\gamma}(1/r)}\,\frac{dr}{r^{n+1}}
        < \infty, \quad\text{for some}\quad \gamma\in (0,1].
     \end{equation}
    Then $\varpi\in\Kato$ with respect to $p^{(\gamma)}_t (x,y)$.

\medskip\textbf{\upshape b)}
    Suppose that $\varpi\in\Kato$ with respect to $p^{(\gamma)}_t(x,y)$, where $\gamma\in (0,1]$. Then
    \begin{equation}\label{s4-e20}
        \lim_{t\to 0} \sup_x \int_0^t \frac{\varpi(B(x,r))}{(q^*)^\gamma(1/r)}\,\frac{dr}{r^{n+1}}
        = 0.
    \end{equation}
\end{lemma}
The next lemma is due to Hawkes \cite[Lemma~2.1]{Ha74}, cf.\ also \cite[Proof of Theorem 1]{Ha71}; it plays the key role in the proof of Theorem~\ref{H-bounds}.
\begin{lemma}\label{ha}
    Let $T^{\gamma}$ be a stable subordinator of index $\gamma\in (0,1)$, and let $B\subset [0,\infty)$ be a Borel set. Then
    \begin{align*}
        \Pp\left(T_t^{\gamma}\in B\quad \text{for some $t>0$}\right)
        = 0 \quad &\text{implies} \quad \dim B \leq 1-\gamma,
    \intertext{while}
        \Pp\left(T_t^{\gamma}\in B\quad \text{for some $t>0$}\right)
        > 0 \quad &\text{implies}\quad \dim B \geq 1-\gamma.
    \end{align*}
\end{lemma}

\bigskip
We are now ready for the
    \begin{proof}[Proof of Theorem~\ref{H-bounds}]
    By assumption, $D$ is a closed $d$-set; pick a corresponding $d$-measure $\varpi$ on $D$.  For $d > n-\alpha$ we have
    \begin{equation*}
        \int_0^1 \sup_x \frac{\varpi(B(x,r))}{q^*(1/r)}\,\frac{dr}{r^{n+1}}
        \leq c_1\int_0^1 \frac{r^d}{q^*(1/r)}\,\frac{dr}{r^{n+1}}
        \leq c_2 \int_0^1 r^{d+\alpha-n-1}\,dr
        < \infty,
    \end{equation*}
    where we used that $q^*(r)\geq c  r^\alpha$, cf.\ \eqref{set-e33}. By Lemma~\ref{lem-SK} (used for $\gamma=1$) we have $\varpi\in\Kato$ w.r.t.\ $p_t(x,y)$, and by Lemma~\ref{lem-dmeas} all points of $D$ are regular for $D=\overline{D}$.

    As $X_0=x\in D$, the set $\{s\,:\, X_s(\omega)\in D\}$ is a.s.\ non-empty, and therefore the random set
    \begin{equation}\label{s4-e30}
    \begin{split}
        W
        := &\left\{ (\omega,\omega^*)\,:\, X_{T^{\gamma}_t(\omega^*)}(\omega) \in D \quad\text{for some $t>0$}\right\}\\
        =  &\left\{ (\omega,\omega^*)\,:\, T_t^{\gamma}(\omega^*)\in \{s\,:\, X_s (\omega)\in D\}\quad \text{for some $t>0$}\right\}
    \end{split}
    \end{equation}
    is well-defined and non-void.

    First we calculate the lower bound of the Hausdorff dimension of the random set $\{ s:\, X_s(\omega)\in D\}$. Assume that  $\gamma\in (\gamma_{\inf},1)$. Recall  that the  transition probability density of  the subordinate process $X_{T^{\gamma}_t(\omega^*)}(\omega)$
    is given by \eqref{s4-e10}. By Lemma~\ref{lem-SK},  $\varpi\in \Kato$ with  respect to $p^{(\gamma)}_t(x,y)$  for any $\gamma\in (\gamma_{\inf},1) $.  Using  Lemma~\ref{lem-dmeas} we see that the points, which are regular for $D$ ``in the eyes'' of the original process $X$, are still regular for $D$ and the subordinate  process $X_{T_t^{\gamma}}$---whenever $\gamma\in (\gamma_{\inf},1)$. This  implies that the set $W$ has full $\Pp^x\otimes \Pp^*$-measure. Thus, \eqref{s4-e30} yields
    $$
        1
        = (\Pp^x\otimes \Pp^*)(W)
        = \int_\Omega \Pp^*\left(\omega^*\::\: T_t^{\gamma}(\omega^*)\in \{ s\,:\, X_s (\omega) \in D\}\quad \text{for some $t>0$}\right) \Pp^x(d\omega),
    $$
    which in turn gives
    $$
        \Pp^x\left(\omega:\, \Pp^*\left[ \omega^*\::\: T_t^{\gamma}(\omega^*)\in \{ s\,:\, X_s(\omega) \in D\}\quad \text{for some $t>0$}\right] > 0\right) = 1.
    $$
    Now Lemma~\ref{ha} shows $\dim \{s: \, X_s(\omega)\in D\}\geq 1-\gamma$ with $\Pp^x$-probability $1$; letting $\gamma\downarrow \gamma_{\inf}$ along a countable sequence we arrive at
    \begin{equation*}
        \dim \{s: \, X_s(\omega)\in D\}
        \geq 1-\gamma_{\inf}\qquad \Pp^x\text{-a.s.}
    \end{equation*}

    To show the upper bound in \eqref{set-e50}, we take  $\gamma\in (0,\gamma_{\sup})$. By the definition of $\gamma_{\sup}$,
    $$
        x\mapsto \int_0^\delta \frac{\varpi(B(x,r))}{(q^*)^\gamma (1/r)}\,\frac{dr}{r^{n+1}}
    $$
    is unbounded for any finite measure $\varpi$ supported in $D$.  There exist, see \eqref{app-e60} below, constants $a,\, b, \delta(T)>0$ such that
    $$
        \int_0^T \int_D p^{(\gamma)}_t(x,y)\,\varpi(dy)\,dt
        \geq a \int_0^{\delta(T)} \frac{\varpi(B(x,r))}{(q^*)^\gamma(1/r)} \,\frac{dr}{r^{n+1}}.
    $$
    Thus, $R_0 \varpi(x)$ is unbounded and, by Remark~\ref{equiv}, the set $D$ is polar for $X_t^{\gamma}$. Therefore, $(\Pp^x\otimes \Pp^*)(W)=0$ and, consequently,
    $$
        \Pp^x\left(\omega\,:\, \Pp^*\left[\omega^*\::\: T_t^{\gamma}(\omega^*)\in \{ s\,:\, X_s^x (\omega) \in D\}\quad \text{for some $t>0$}\right] =0 \right) = 1.
    $$
    This means that $\{s\,:\, X_s^x (\omega) \in D\}$ is polar for $T_t^{\gamma}$ with $\Pp^x$-probability $1$.  Applying Lemma~\ref{ha} we get
    $\dim  \{s: \, X_s^x(\omega)\in D\}\leq 1-\gamma$ with $\Pp^x$-probability $1$. Letting $\gamma \uparrow \gamma_{\sup}$ along a countable sequence, the upper bound in \eqref{set-e50} follows.
\end{proof}

\bigskip
\begin{proof}[Proof of Proposition~\ref{prop-H2}]
    Since the processes $X^1$ and $X^2$ are, up to different starting points, i.i.d.\ copies, the transition probability density of $\tilde{X}_t:= X^1_t-X_t^2$ is given by
    $$
        \tilde{p}_t(x,y)
        = \int_\real p_t(x+x_0,z+y) p_t(x_0,z)\,dz;
    $$
    here $x_0\in \real$ is the starting point of $X_t^2$. Let us estimate $\tilde{p}_t(x,y)$ using the upper bounds \eqref{set-e36} for $p_t(x,y)$.  By the triangle inequality we have for any $\epsilon>0$ and $w_1,w_2\in\real$
    \begin{align*}
        \int_\real \rho_t^2 &e^{-a_4\rho_t \,|z+y-x-x_0-w_1|} e^{-a_4\rho_t\,|x_0-z+w_2|} \,dz\\
        &\leq \rho_t\, e^{-a_4\epsilon \rho_t \,|y-x-w_1+w_2|} \int_\real \rho_t \,e^{-a_4(1-\epsilon)  \rho_t\cdot ( |z+y-x-x_0-w_1|+|x_0-z+w_2|)}\,dz \\
        &\leq c \rho_t \,e^{-a_4\epsilon \rho_t\,|y-x-w_1+w_2|}.
    \end{align*}
    This yields the following upper bound for $\tilde{p}_t(x,y)$:
    \begin{equation*}
    \begin{split}
        \tilde{p}_t(x,y)
        &\leq a_3^2 \iiint_{\real^3} \rho_t^2 e^{-a_4\rho_t \,|z-x-x_0-w_1|} e^{-a_4\rho_t\,|z-x_0-w_2|} \,dz \,Q_t(dw_1) \,Q_t(dw_2)\\
        &\leq C \rho_t \left(f_{\textrm{up}}^{\epsilon} (\rho_t \cdot\, ) * \tilde{Q}_t\right)(y-x),
    \end{split}
    \end{equation*}
    where $\tilde{Q}_t(dw):= \int_\real Q_t(dw+v)\,Q_t(dv)$ is again a sub-probability measure. In other words, the transition probability density of $\tilde{X}$ has an upper bound of the same form as $p_t(x,y)$.

    To show the lower bound, take $x,y$ such that $\rho_t \,|y-x|\leq a_2^{-1}(1-a_2 \epsilon)$, where $\epsilon>0$ is small. Then
    \begin{align*}
        \tilde{p}_t(x,y)
        &\geq a_1^2 \rho_t^2 \int_\real f_{\textrm{low}}((y+z-x-x_0)\rho_t)f_{\textrm{low}}((x_0-z)\rho_t)\,dz\\
        &\geq a_1\rho_t \int_{|v|\leq \epsilon} f_{\textrm{low}} (\rho_t (y-x-v/\rho_t))f_{\textrm{low}}(v)\,dv.
    \end{align*}
    Since for $|v|\leq\epsilon$
    $$
        1-a_2 \rho_t \,|y-x-v/\rho_t|
        \geq 1-a_2 \epsilon -a_2 \rho_t\, |y-x|
        = (1-\epsilon a_2)\left( 1- \frac{a_2}{1-a_2\epsilon} \rho_t \,|y-x|\right),
    $$
    we get for all $x,y$ such that $\rho_t \,|y-x|\leq a_2^{-1}(1-a_2 \epsilon)$ the estimate
    $$
        \tilde{p}_t(x,y)
        \geq c_1 \rho_t\left( 1- c_2 \rho_t\,|y-x|\right)
    $$
    with $c_1= a_1 (1-a_2\epsilon) \int_{|v|\leq\epsilon} f_{\textrm{low}}(v)\,dv$ and $c_2= a_2 (1-a_2\epsilon)^{-1}$. Thus, the lower bound for $\tilde{p}_t(x,y)$ is also of the same form as the one for $p_t(x,y)$.

    We have shown that the symmetrized process $\tilde X$ satisfies the estimates \eqref{set-e35} and \eqref{set-e36}, and these estimates are the essential ingredient in the proof of Corollary~\ref{cor-H}.\footnote{Notice that Assumption~A in Corollary~\ref{cor-H} is just used to ensure that we have \eqref{set-e35} and \eqref{set-e36}.}. Thus, we can apply Corollary~\ref{cor-H}, and the proof is finished.
\end{proof}

\section{Proof of Theorem~\ref{t-col}}\label{s5}

Throughout this section we work under the assumptions of Theorem~\ref{t-col}: Assumption~A holds and the process $X$ is recurrent. We denote the transition probability density by $p_t(x,y)$, $t>0$. Recall that the process $X$ is called
\begin{enumerate}
\item[1)] (\emph{neighbourhood}) \emph{recurrent} if
$$
    \forall x\in \real\quad
    \forall \text{open sets\ } G\subset\real
    \::\: \Pp^x \left(X_t\in G\quad \text{for some $t>0$}\right) = 1.
$$
\item[2)] \emph{point recurrent}, if
$$
    \forall x,y\in\real\::\: \Pp^x\left(X_t=y\quad \text{for some $t>0$}\right) = 1.
$$
\end{enumerate}
Using the arguments of \cite[Lemma~4.1]{JP69} we can show that, in the setting of Theorem~\ref{t-col}, the recurrence of $X$ already implies point recurrence.
\begin{lemma}\label{lem-rec}
    The process $X$ is point recurrent.
\end{lemma}
\begin{proof}
    Write $\tau_y:=\inf\{ t>0\,:\, X_t=y\}$ for the hitting time of $\{y\}$ and set
\begin{equation*}
    \Phi(x,y):= \Pp^x\left(X_t =y \quad\text{for some $t>0$}\right)
    = \Ee^x \I_{\{\tau_y<\infty\}}.
\end{equation*}
Let us show that for $X$ any singleton $\{x\}$ is regular for itself. By \eqref{set-e36} and the inequality $\rho_t\leq c t^{-1/\alpha}$, $t\in (0,1]$---this follows from \eqref{set-e33}---we have for $\alpha>1$
 \begin{align*}
     \sup_{x,y\in \real}\int_0^t p_s(x,y)\,ds
     \leq c_1 \int_0^t \rho_s \,ds
     \leq c_2 \int_0^t s^{-1/\alpha}\,ds
     \leq c_3 t^{1-1/\alpha}, t\in (0,1].
\end{align*}
    Thus, any measure of the form $\varpi = c\delta_y$ for $c\geq 0$ and some $y$  belongs to the Kato class $\Kato$ w.r.t.\ $p_t(x,y)$. By Lemma~\ref{lem-dmeas}, any point $y\in \real$ is regular for itself for $X$. Then
\begin{equation*}%\label{s5-e10}
    \Phi(y,y)=1,
\end{equation*}
    because $\{\tau_x=0\} = \bigcap_{\epsilon>0} \{X_t=x \text{\ \ for some\ \ } t\in (0,\epsilon)\}$, and by regularity $\Pp^x(\tau_x=0) =1$.

    Let us show that the function $\Phi(\cdot,y)$ is excessive. Denote by $(P_t)_{t\geq 0}$ the semigroup given by the kernel $p_t(x,y)$. Since
\begin{align*}
    \Phi(X_t(\omega),y)
    &= \Pp^{X_t (\omega)}\left(X_s=y \quad\text{for some $s>0$}\right)\qquad \text{for $\Pp^x$-a.a.\ $\omega$}\\
    &= \Pp^x\left(X_{t+s}=y \quad\text{for some $s>0$}\right),
\end{align*}
    we have
$$
    P_t\Phi(\cdot,y)(x)
    = \Ee^x \Phi(X_t,y)
    = \Pp^x\left(X_{t+s}=y \quad\text{for some $s>0$}\right)
    \leq \Phi(x,y),
$$
    and by the dominated convergence theorem  $P_t  \Phi(x,y) \uparrow \Phi(x,y)$ as $t\to 0$. Since $X$ is recurrent, all excessive functions are constant, see \cite[Exercise~10.39]{Sh88}; hence, we get $\Phi(x,y)\equiv 1$ for all $x,y\in \real$.
\end{proof}

\begin{remark}\label{point-rec}
    Let $X^1$ and  $X^2$ be two independent copies of $X$. Then the symmetrized process $\tilde{X}= X^1 -X^2$ is point recurrent.
\end{remark}

Let $\beta$ be the exponent appearing in the upper bound in \eqref{set-e70}.
\begin{lemma}\label{t.3.1}
    Let $X^1$ and $X^2$ be independent copies of $X$, and denote by $Z^\beta$ a symmetric $\beta$-stable L\'evy process in $\real^2$. Let $D$ be a subset of the diagonal  in $\real^2$. If $D$ is polar for  $Z^\beta$, then it is polar for $(X^1, X^2)$.
\end{lemma}
\begin{proof}
    Denote by $\mathfrak{p}_t(x,y)$, $x=(x_1,x_2)$, $y=(y_1,y_2)$, the transition probability density of the bivariate process $(X^1, X^2)$.   Suppose that $|x-y|\leq \epsilon$ for some sufficiently small $\epsilon>0$. Using the lower estimates \eqref{set-e35} for $p_t(x_i,y_i)$, $i=1,2$, we get
\begin{align*}
    \int_0^1 \mathfrak{p}_t(x,y)\,dt
    &\geq a_1^2 \int_0^1  (1-a_2 \rho_t\,|x_1-y_1|)_+ (1-a_2 \rho_t\,|x_2-y_2|)_+ \rho_t^2 \,dt\\
    &\geq a_1^2 (1-a_2)^2 \int^1_{c \phi(|x-y|)} \rho_t^2 \,dt,
\end{align*}
    where $\phi(r):= 1/q^U(1/r)$; with this choice of $\phi(r)$ we have $\rho_t\,|x-y|<1$. Changing variables gives
\begin{align*}
    \int_0^1 \mathfrak{p}_t(x,y)\,dt
    \geq c_1 \int_{|x-y|}^{1/\rho_1} \frac{1}{r^3 q^U(1/r)}\,dr
    \geq c_2 |x-y|^{\beta-2},
\end{align*}
    where we used the definition of $\rho_t$ as inverse of $q^*$ and \eqref{set-e30}, as well as $\big(q^U (r)\big)'= 2 q^L(r)/ r$ a.e.\ and \eqref{set-e70}. The expression in the last line is (up to a constant) the potential of the process  $Z^\beta$. Thus, for $|x-y|<\epsilon$ the potential of $(X^1,X^2)$ is bounded from below by the potential $U(x):= |x|^{\beta-2}$ of $Z^\beta$. Now
\begin{equation}\label{s5-e18}
    \int_{|x-y|>\epsilon} \frac{1}{|x-y|^{2-\beta}}\, \varpi(dy)
    \leq  \epsilon^{\beta-2}\varpi (D)
    \quad\text{for all finite measures $\varpi$.}
\end{equation}
    By Remark~\ref{equiv} the set $D$ is polar for $Z^\beta$ if and only if the potential of  $Z^\beta$ is unbounded for any finite measure $\varpi\neq 0$ with $\supp\varpi\subset\overline{D}$, i.e.
\begin{equation*}%\label{s5-e20}
    \sup_x U\varpi(x)
    = \sup_x \int  \frac{1}{|x-y|^{2-\beta}}\,\varpi(dy)
    =\infty.
\end{equation*}
    Because of \eqref{s5-e18} this happens if and only if
\begin{equation}\label{s5-e07}
    \sup_x \int_{|x-y|\leq \epsilon} \frac{1}{ |x-y|^{2-\beta}}\, \varpi(dy)=\infty
\end{equation}
    Thus, if \eqref{s5-e07} holds true, then $\sup_x \mathfrak{R}_0 \varpi(x)=\infty$, where $\mathfrak{R}_0$ is the $0$-resolvent for $(X^1,X^2)$; by  Remark~\ref{equiv} the set $D$ is polar for $(X^1,X^2)$.
\end{proof}

The next lemma is from \cite[Theorem~4]{Ta66}, see also \cite{JP69}, and it plays the key role in the proof of Theorem~\ref{t-col}.
\begin{lemma}\label{l.4.2}
    Suppose that $A$ is an analytic subset of $\rn (n=1,2)$, and $Z_t^{\zeta,n}$ is any symmetric $\zeta$-stable L\'evy process in $\rn$. Then
    $$
        \dim A
        = n-\inf\left\{\zeta\,:\, \text{$A$ is non-polar for $Z^{\zeta,n}$}\right\}.
    $$
\end{lemma}

\medskip
\begin{proof}[Proof of Theorem~\ref{t-col}]
    Let $A(\omega)$ be the collision set defined in \eqref{set-e65}. Since the one-dimensional process $X^1- X^2$ is point recurrent, cf.\ Remark~\ref{point-rec}, the set $A(\omega)$ is a.s.\ non-empty. Instead of $A(\omega)$ we consider the following set on the diagonal of $\real^2$:
\begin{align*}
    \hat{A}(\omega)
    := &\left\{(x,x)\in\real^2\,:\,(X^1_t(\omega),X_t^2(\omega))=(x,x)\quad \text{for some $t>0$}\right\}\\
    \equiv &\left\{(x,x)\in\real^2\,:\, \tau^x(\omega)<\infty\right\},
\end{align*}
    where $\tau^x:= \inf\{t>0\,:\,(X^1_t,X_t^2)=(x,x)\}$. There is a one-to-one correspondence between $\hat{A}(\omega)$ and $A(\omega)$, and their Hausdorff dimensions coincide. For our needs it is more convenient to work with the set $\hat{A}(\omega)$.

    Define on a further probability space $(\Omega', \mathcal{F}', \Pp')$ a symmetric $\theta$-stable L\'evy process $Z^{\theta,1}_t (\omega')$, $t\geq 0$, taking values on the diagonal of $\real^2$ and with $\theta<2-\beta$. We are going to show that
\begin{equation}\label{s5-e30}
    \Pp'\left(\omega'\,:\, Z^{\theta,1}_t(\omega')\in \hat{A}(\omega)\quad\text{for some $t>0$}\right) = 0
\end{equation}
    for almost all $\omega$; this means that $\hat{A}(\omega)$ is a.s.\ polar for $Z^{\theta,1}(\omega')$.

    Let
$$
    \Gamma :=
    \left\{(\omega,\omega')\,:\, Z_t^{\theta,1} (\omega') \in \hat{A}(\omega) \quad \text{for some $t>0$}\right\}.
$$
    Then, by the definition of $\hat{A}(\omega)$,
$$
    \Gamma =
    \left\{(\omega,\omega') \,:\, (X_t^1(\omega),X_t^2(\omega))=(x,x)\in \hat{B}(\omega') \quad\text{for some $t>0$}\right\},
$$
    where $\hat{B}(\omega'):= \textrm{Range}\,Z_t^{\theta,1}(\omega')$. In \cite{BG60} it is shown that $\dim \hat{B}(\omega')=\theta$; by Lemma~\ref{l.4.2} we get
$$
    2 - \inf\left\{\zeta>0\,:\, \text{$\hat{B}(\omega')$ is non-polar for $Z^{\zeta,2}$}\right\}
    = \dim \hat{B}(\omega')
    = \theta < 2-\beta,
$$
    and so
$$
    \beta
    < \inf\left\{\zeta>0 \,:\, \text{$\hat{B}(\omega')$ is non-polar for $Z^{\zeta,2}$}\right\}.
$$
    Thus, the set $\hat{B}(\omega')$ is for almost all $\omega'$  polar for the process $Z_t^{\beta,2}$. By Lemma~\ref{t.3.1} the set $\hat{B}(\omega')$ is polar for $(X_t^1(\omega),X_t^2(\omega))$ for almost all $\omega'$. By Fubini's theorem  we have $\Pp\otimes\Pp' (\Gamma)=0$; therefore, \eqref{s5-e30} holds true, showing that $\hat{A}(\omega)$ is polar for $Z^{\theta, 1}$ for all $\theta<2-\beta$. Thus, by Lemma~\ref{l.4.2}
$$
    \dim \hat{A}(\omega)
    = 1-\inf\left\{\theta>0\,:\, \text{$\hat{A}(\omega)$ is non-polar for $Z^{\theta, 1}$}\right\}
    \leq 1- (2-\beta)
    = \beta-1.
$$

    Next, we are going to show that $\dim\hat{A}(\omega) \geq \alpha-1$. Choose $\theta\in (2-\alpha, 2)$, and let $Z^{\theta,1}$  be a symmetric $\theta$-stable L\'evy process on the diagonal in $\real^2$. Denote by $\hat{B}(\omega')$ its range; by \cite{BG60}, $\dim \hat{B}(\omega')=\theta$. By  Frostman's lemma, cf.\ e.g.\ \cite[p.~387, Theorem A.44]{sch-bm}, there exists a measure $\varpi$ on $\hat{B}(\omega')\cap K$ ($K$ is a compact subset of the diagonal in $\real^2$) such that
\begin{equation}\label{s5-e35}
        \varpi\big(B(z,r)\big)
        \leq Cr^{\theta-\epsilon},
        \quad z\in \hat{B}(\omega'), \; r>0.
\end{equation}
    Denote by $\mathfrak{p}_t(x,y)$ the transition probability density of $(X^1_t,X^2_t)$.  A direct calculation shows (cf.\ \eqref{app-e20} in the appendix for details of the first estimate) that \eqref{s5-e35} implies
\begin{align*}
    \int_0^1 \int_{\hat{B}(\omega')\cap K} \mathfrak{p}_t(x,y)\,\varpi(dy)\,dt
    &\leq c_1 \int_0^1 \int_0^\infty \rho_t^2  \sup_{x\in \real^2} \varpi\{y\,:\, |x-y|\leq c_2 r/\rho_t\}\, e^{-r}\,dr\,dt\\
    &\leq c_3 \int_0^1 \rho_t^{2-\theta+\epsilon}\,dt\\
    &\leq c_4 \int_0^1 t^{-(\theta-2+\epsilon)/\alpha}\,dt < \infty,
\end{align*}
which shows that $\varpi\in \Kato$ w.r.t.\ $\mathfrak{p}_t(x,y)$. Hence, by Corollary~\ref{poten} the set $\hat{B}(\omega')$ is non-polar for $(X^1,X^2)$.

By $\Pp^{(z,z)}$ we indicate that the starting point of the process $(X^1,X^2)$ is $(z,z)$.  For all $(z,z)\in \real^2$
$$
    \Pp^{(z,z)}\otimes \Pp'\left( (\omega,\omega') \,:\, (X_t^1(\omega),X_t^2(\omega))=(x,x)\in \hat{B}(\omega') \quad\text{for some $t>0$}\right) > 0.
$$
By Fubini's theorem, there is a set $F\in \mathcal{F}$ with $\Pp^{(z,z)}(F)>0$ such that
\begin{equation*}%\label{s5-e40}
    \forall\omega\in F\::\: \Pp'\left(\omega' \,:\, Z_t^{\theta,1}(\omega') \in \hat{A}(\omega)\quad \text{for some $t>0$}\right) > 0.
\end{equation*}
Let us show that there exists an open subset $\mathcal{O}$ of the diagonal in $\real^2$ such that
\begin{equation}\label{s5-e45}
    \inf_{(z,z)\in \mathcal{O}} \Pp^{(z,z)}(F)
    \geq \delta > 0.
\end{equation}
Indeed, for any $s>0$ we have
\begin{align*}
    \Ee^{(z,z)} \Ee'\I_{\{\exists t>0\,:\, (X_t^1,X_t^2)\in \hat{B}(\omega')\}}
    &\geq \Ee^{(z,z)} \Ee'\I_{\{\exists t>s\,:\, (X_t^1,X_t^2)\in \hat{B}(\omega')\}}\\
    &= \Ee^{(z,z)} \left(\Ee^{(X_s^1,X_s^2)} \Ee'\I_{\{\exists t>s\,:\, (X_{t-s}^1,X_{t-s}^2)\in \hat{B}(\omega')\}}\right)\\
    &= \Ee^{(z,z)} \left(\Ee^{(X_0^1,X_0^2)} \Ee'\I_{\{\exists t>s\,:\, (X_{t-s}^1\circ \theta_s,X_{t-s}^2\circ\theta_s)\in \hat{B}(\omega')\}}\right)\\
    &=\Ee^{(z,z)} \left(\Ee'\I_{\{\exists r>0\,:\, (X_{r}^1\circ \theta_s,X_{r}^2\circ\theta_s)\in \hat{B}(\omega')\}}\right).
\end{align*}
Denote by $\theta_s^{-1} F = \{\omega\,:\, \theta_s\omega\in F\}$ the shift of the set $F$. Clearly, $\theta_s^{-1}\{Y\in \Gamma\} = \{Y\circ\theta_s\in\Gamma\}$ for any random variable $Y(\omega)$ and a set $\Gamma$ in the state space of $Y$, and so
$$
    \Pp^{(z,z)}(F)\geq \Pp^{(z,z)}(\theta_s^{-1}F).
$$
This inequality implies that $\Pp^{(z,z)}(F)$ is excessive:
\begin{align*}
    P_t \Pp^{(z,z)}(F)
    = \Ee^{(z,z)} \Pp^{(X_t,X_t)}(F)
    &= \Ee^{(z,z)} \left( \Ee^{(z,z)} \left[\I_{\theta_t^{-1} F} \:\middle|\:\mathcal{F}_t\right] \right)\\
    &= \Ee^{(z,z)} \I_{\theta_t^{-1} F}
    = \Pp^{(z,z)}(\theta_t^{-1} F)
    \leq \Pp^{(z,z)}(F),
\end{align*}
and, by the dominated convergence theorem, $P_t \Pp^{(z,z)}(F) \uparrow \Pp^{(z,z)}(F)$ as $t\to 0$.  Assumption~A implies that the process $X$ is a strong Feller process, cf.\ Section~\ref{set}, which means that any excessive function is lower semicontinuous, see \cite[p.~77, Exercise~2.16]{BG68}.  Since $z\mapsto \Pp^{(z,z)}(F)$ is lower semi-continuous we get \eqref{s5-e45}.

Fix $\epsilon>0$, and define
$$
    \tau_1:= \inf\{ t>\epsilon\,:\,X^1_t=X_t^2 \}, \quad
    \tau_1^x := \inf\{ t>\epsilon: \quad X^1_t=X^2_t=x\}, \quad x\in\real,
$$
and
$$
    \hat{A}_1(\omega):= \left\{(x,x)\in\real^2\,:\,\tau_1^x(\omega) <\infty\right\}.
$$
We have
$$
    \Pp'\left(\omega'\,:\, Z_t^{\theta,1} (\omega')\in \hat{A}_1(\omega)\quad\text{for some $t\in (0,\tau_1]$}\right)
    > 0,  \quad \forall\omega\in F.
$$
Thus,
$$
    \inf\left\{\zeta\,:\, \hat{A}_1(\omega) \quad \text{is non-polar for $Z^{\zeta,1}$}\right\}
    \leq 2-\alpha,
$$
which implies by Lemma~\ref{l.4.2}  that
$$
    \dim \hat{A}_1(\omega)
    = 1 - \inf\left\{\zeta \,:\, \hat{A}_1(\omega) \quad\text{is non-polar for $Z^{\zeta,1}$}\right\}
    \geq 1 - (2-\alpha)
    = \alpha-1,\quad\forall\omega\in F.
$$

For $(z,z)\in \mathcal{O}$ we get
\begin{equation}\label{s5-e50}
    \Pp^{(z,z)}\left(\dim\hat{A}_1(\cdot)\geq \alpha-1\right)
    \geq \Pp^{(z,z)}(F)\geq \delta>0.
\end{equation}

Let us now show that
$$
    \dim \hat{A}(\omega)
    \geq \alpha-1\quad \Pp^{(z,z)}\text{-a.e.\ for all $z\in \real$}.
$$
Let $\tau_0(\omega)=0$ and define
$$
    \tau_n(\omega)
    := \inf\left\{ t>\tau_{n-1}(\omega) + \epsilon \,:\, (X^1_t(\omega),X^2_t(\omega))=(x,x)\in K\right\},
$$
where $K$ is as above. Since the process $X^1-X^2$ is point recurrent, the stopping times $\tau_n$ are almost surely finite. Define $G_1(\omega):= \dim \hat{A}_1(\omega)$, and
$$
    G_n(\omega):=
    \dim\left\{(x,x)\in\real^2 \,:\, X^1_t = X^2_t = x\quad\text{for some $t\in (\tau_{n-1}(\omega), \tau_{n}(\omega)]$}\right\}, \quad n\geq 2.
$$
Note that $\dim \hat{A}(\omega)\geq \sup_n G_n(\omega)$. Using the Markov property and \eqref{s5-e50} we get
\begin{align*}
    \Pp^{(z,z)}\left(\dim \hat{ A} <1-\theta\right)
    &\leq \Pp^{(z,z)}\left(\sup_n G_n<1-\theta\right)\\
    &\leq \Pp^{(z,z)}\left(\max_{i\leq n} G_i  <1-\theta\right)\\
    &= \Ee^{(z,z)}\left(\I_{\{ \max_{1\leq i\leq n-1} G_i<1-\theta\}} \vphantom{\Ee^{(X_{\tau_{n-1}}^1, X_{\tau_{n-1}}^2)}}\Ee^{(z,z)}\left[ \I_{\{ G_n<1-\theta\}} \:\middle|\: \mathcal{F}_{\tau_{n-1}} \right] \right)\\
    &= \Ee^{(z,z)}\left(\I_{\{ \max_{1\leq i\leq n-1} G_i <1-\theta\}} \Ee^{(X_{\tau_{n-1}}^1, X_{\tau_{n-1}}^2)}\left[\I_{\{ G_1 <1-\theta\}}\right] \right)\\
    &\leq (1-\delta) \Ee^{(z,z)}\left(\I_{\{ \max_{1\leq i\leq n-1} G_i <1-\theta\}} \right)\\
    &\leq (1-\delta)^n
\end{align*}
for all $n\geq 1$.  Therefore, $\dim A(\omega)\geq 1-\theta$ $\Pp$-a.s. Letting $\theta\downarrow 2-\alpha$ along a countable sequence, the claim follows.
\end{proof}

\section{Examples: Recurrent processes satisfying Assumption~A}\label{rec}

In this section we give examples of processes $X$ which satisfy Assumption~A and which are recurrent. For simplicity, we will assume that the space dimension $n=1$.

\begin{example}\label{exa1}
Let
\begin{equation*}%\label{rec-e05}
    j(x,u):= \big( n(x,u)+n(u,x) \big) \mathfrak{g}(x-u),
\end{equation*}
where the function $n(x,u)$ is strictly  positive, uniformly bounded and H\"older continuous in both variables,  and  $\mathfrak{g}$  satisfies \eqref{set-e38}. In particular, \eqref{set-e38} holds true for $\mu(du)=\mathfrak{g}(u)du$, if

\begin{enumerate}
\item[a)]
    $\mathfrak{g}$ is even and $\mathfrak{g}(h)\leq C |h|^{-1-\eta}$ for some  $\eta>1$ and all $|u|\geq 1$;
\item[ b)]
    There exists some $\epsilon\in (0,1)$ such that $h^{2+\epsilon}\mathfrak{g}(h)$ is increasing on $(0,1]$;
\item[c)]
    There exits some $\delta>1$ such that the function $h^\delta \mathfrak{g}(h) $ is decreasing on $(0,1]$.
\end{enumerate}
Let us check \eqref{set-e38}. As in Section~\ref{set} we write
\begin{equation*}%\label{rec-e10}
    q(\xi)
    = \int_{\real\setminus\{0\}} \big( 1- \cos(\xi h) \big) \mathfrak{g}(h)\,dh
\end{equation*}
and we define the corresponding upper and lower symbols $q^U(\xi)$ and $q^L(\xi)$.  The conditions a)--c)   imply \eqref{set-e30}.  Indeed, by b) we have for $|\xi|\geq 1$
$$
    q^L(\xi)
    = \frac{2}{\xi}  \int_0^{1} v^2 \mathfrak{g}\left(\tfrac v\xi\right)\,dv
    = \frac{2}{\xi}  \int_0^{1} v^{-\epsilon} v^{2+\epsilon} \mathfrak{g}\left(\tfrac v\xi\right)\,dv
    \leq \frac{c_1}{\xi}\mathfrak{g}\left( \tfrac{1}{\xi}\right),
$$
and by c) we get
\begin{gather*}
    q^{L}(\xi)
    = \frac{2}{\xi} \int_0^{1} v^2 \mathfrak{g}\left(\tfrac v\xi\right)\,dv
    \geq \frac{2}{\xi} \int_0^{1} v^{2-\delta} v^\delta \mathfrak{g}\left(\tfrac v\xi\right)\,dv
    \geq  \frac{c_2}{\xi}\mathfrak{g}\left( \tfrac{1}{\xi}\right),\\
    \int_{1/\xi}^1 \mathfrak{g}(h)\,dh
    = \int_{1/\xi}^1 h^{-\delta} h^{\delta} \mathfrak{g}(h)\,dh
    \leq \frac{c_3}{\xi}\mathfrak{g}\left( \tfrac{1}{\xi}\right),
\end{gather*}
which implies \eqref{set-e30}. Since $(1-\cos 1) q^L(\xi) \leq q(\xi) \leq 2 q^U(\xi)$, we have  $q(\xi)\asymp |\xi|^{-1}\mathfrak{g}\left(|\xi|^{-1}\right)$ for large $|\xi|$.

Note that the estimate a) gives
\begin{equation*}%\label{rec-e20}
\begin{split}
    q(\xi)
    &= \int_{|h|\leq 1} (1-\cos \xi h)\,\mathfrak{g}(h)\,dh + \int_{|h|\geq 1} (1-\cos \xi h)\mathfrak{g}(h)\,dh\\
    &\leq c_1\xi^2 + c_2 |\xi|^\eta \int_{\xi}^\infty (1-\cos v)\frac{dv}{v^{1+\eta}}\\
    &\leq c_3 |\xi|^{2\wedge \eta},  \quad  \quad |\xi|\leq 1.
\end{split}
\end{equation*}
Assume also that
\begin{equation*}%\label{rec-e30}
    \int_{|h|\leq 1}  |h|   \big| j(x,x+h)- j(x,x-h)\big| \,dh <\infty.
\end{equation*}
Then the function
\begin{equation*}%\label{rec-e40}
    k(x)
    := \frac{1}{2} \int_{|h|\leq 1}  \big( j(x,x+h)- j(x,x-h)\big) h\, dh
\end{equation*}
is well-defined.

Consider the operator $\mathcal{L}$ defined by \eqref{set-e05} with $a(x)=k(x)$, and
\begin{equation*}%\label{rec-e50}
    m(x,h)\,\mu(dh)
    = j(x,x+h)\, dh
    = \big( n(x,x+h)+ n(x+h,x)\big) \mathfrak{g}(h)\,dh.
\end{equation*}
Then
\begin{equation*}%\label{rec-e60}
    \mathcal{L} f(x) =  \int_{\real\setminus\{0\}} \big(f(x+h)-f(x)\big) j(x,x+h)\, dh,
\end{equation*}
which is a symmetric operator, and  generates a (symmetric) Dirichlet form
$$
    \mathcal{E}(\phi,\phi)
    =\frac{1}{2} \int_{\real\setminus\{0\}} \int_\real (\phi(x+h)-\phi(x))^2\, j(x,x+h)\,dx\,dh.
$$
The form $\mathcal{E}(\cdot,\cdot)$ is comparable with the Dirichlet form $\mathcal E^q(\cdot,\cdot)$ corresponding to the L\'evy process $Z$ with characteristic exponent $q$:
\begin{align*}
    \mathcal{E}(\phi,\phi)
    &\asymp \mathcal{E}^{q}(\phi,\phi)
    := \int_\real q(\xi) |\hat{\phi}(\xi)|^2d\xi
    = \frac 12\int_\real \int_{\real\setminus\{0\}}  \big(\phi(x+h)-\phi(x)\big)^2 \mathfrak{g}(h)\,dh\, dx.
\end{align*}
The  Dirichlet form $\mathcal{E}^{q}$ is recurrent, because the related L\'evy process $Z$ is recurrent by the Chung-Fuchs criterion, i.e.\  $\int_{|\xi|\leq 1} q(\xi)^{-1}\,d\xi =\infty$. By Oshima's criterion, cf.\ \cite{Osh92}, the form $\mathcal{E}$ is also recurrent implying the recurrence of the related process $X$.
\end{example}

\begin{example}\label{exa2}
    Let $\mathcal{L}$ be the generator defined by \eqref{set-e05}. In order to construct an example of a non-symmetric recurrent Markov process satisfying Assumption~A, we use the approach from \cite{W08}. Note that our Assumption~A implies the condition (H) needed in \cite{W08}. According to \cite[Theorem~1.4]{W08} the following condition is sufficient for the recurrence of the process $X$:
    \begin{equation}\label{rec-e70}
        B(x)x + D(x)|x|\leq C \quad \text{for sufficiently large $|x|$,}
    \end{equation}
    where $B(x):= b(x)+ \int_{1<|z|\leq |x|} z m(x,z)\,\mu(dz)$ and $D(x):= \int_{|z|\geq |x|} |z| m(x,z)\,\mu(dz)$, with $b(x)$ and $m(x,u)$ from the representation of $\mathcal{L}$ in \eqref{set-e05}.  Thus, by  \cite[Theorem~1.4]{W08}, the process $X$ which corresponds to \eqref{set-e05} is recurrent, if \eqref{rec-e70} holds true.
\end{example}

\section{Appendix}\label{app}

\subsection*{Proof of  Lemma~\ref{lem-dmeas}}
Without loss of generality we may assume that $D$ is a closed set.
We begin with the upper bound for
$$
    R_\lambda\varpi(x)
    = \int_0^1 \int_{D} e^{-\lambda t} p_t(x,y) \,\varpi(dy)\,dt + \int_1^\infty \int_{D} e^{-\lambda t} p_t(x,y)\,\varpi(dy)\,dt.
$$
The upper estimate for the second term can be proved in the same way as  \cite[(3.3)]{KT06}: for any $x\in D$ and $\lambda>0$ one finds that
\begin{align*}
    \int_1^\infty \int_D e^{-\lambda t} p_t(x,y)\,\varpi(dy)\,dt
    \leq \frac{e^{-\lambda }}{1-e^{-\lambda }}   \sup_{x\in D} \int_0^1 \int_D p_s(x,y)\,\varpi(dy)\,ds,
\end{align*}
where we used in the last line that the integral on the right-hand side is finite since $\varpi\in\Kato$.
Therefore,
\begin{equation*}%\label{app-e05}
    \int_0^1 \int_D e^{-\lambda t} p_t(x,y)\,\varpi(dy)\,dt
    \leq R_\lambda \varpi(x)
    \leq \left( 1+ \frac{e^{-\lambda}}{1-e^{-\lambda}}\right)\sup_{x\in D} \int_0^1 \int_D e^{-\lambda t} p_t(x,y) \,\varpi(dy)\,dt.
\end{equation*}

Using the upper and lower bounds \eqref{set-e35}, \eqref{set-e36} for the heat kernel,  we obtain by a change of variables
\begin{align*}
    \int_0^1 \int_D &e^{-\lambda t} p_t(x,y)\,\varpi(dy)\,dt \\
    &\geq  a_1 \int_0^1 \int_D e^{-\lambda t}\rho_t^n (1-a_2 |x-y|\rho_t)_+ \,\varpi(dy) \,dt\\
% &=a_1 \int_0^1 \int_0^1 e^{-\lambda t}\rho_t^n \varpi\{y: (1-a_2 |x-y|\rho_t)_+ \geq r\} dr dt\\
    &= a_1 \int_0^1 \int_0^1 e^{-\lambda t} \rho_t^n \,\varpi\left\{y\in D\,:\, (1-a_2 |x-y|\rho_t)_+ \geq 1-r\right\} \,dr \,dt\\
    &= a_1 \int_0^1 \int_0^1 e^{-\lambda t}\rho_t^n \,\varpi\left\{y\in D\,:\, a_2 |x-y|\rho_t\leq  r\right\} \,dr \,dt.
\end{align*}
Using the lower bound in \eqref{set-e45} for the $d$-measure $\varpi$, we get for $x\in D$
\begin{equation*}
    \int_0^1 \int_D e^{-\lambda t} p_t(x,y)\,\varpi(dy)\,dt
    \geq  c_1 a_2^{-d} \int_0^1 \int_0^1 \rho_t^{n-d} e^{-\lambda t} r^d\, dt\,  dr
    = c  \lambda^{-1} \int_0^\lambda e^{-u} \rho_{u/\lambda}^{n-d}\,du.
\end{equation*}
Similarly, we have with the upper bound \eqref{set-e36},
\begin{equation}\label{app-e20}
\begin{split}
    \int_0^1 \int_D &e^{-\lambda t} p_t(x,y) \,\varpi(dy) \,dt\\
    &\leq a_3 \int_0^1 \int_\rn \int_D e^{-\lambda t}\rho_t^n e^{-a_4|x-y-z|\rho_t}\,\varpi(dy) \,Q_t(dz) \,dt\\
    &\leq a_3 \int_0^1 \int_\rn \int_0^\infty e^{-\lambda t} \rho_t^n \,\varpi\left\{y\in D\,:\, a_4 |x-y-z|\rho_t\leq r\right\} e^{-r} \,dr \,Q_t(dz) \,dt\\
    &\leq  c_4  \int_0^1 \int_0^\infty e^{-t\lambda} \rho_t^n\, \sup_{w\in \rn} \varpi\left\{ y\in D\,:\, a_4|w-y|\rho_t \leq r\right\} e^{-r} \,dr\, dt\\
    &\leq c_5 \int_0^1  e^{-t\lambda} \rho_t^{n-d}  \,dt \cdot \int_0^\infty r^d  e^{-r} \,dr\\
    &= c_6  \lambda^{-1} \int_0^\lambda e^{-u} \rho_{u/\lambda}^{n-d}\,du.
\end{split}
\end{equation}
Here we used the upper bound \eqref{set-e45} for small $r$, and the fact that $\supp\varpi=D$, which implies $\sup_x \varpi(B(x,r))\leq C$ for large  $r$. This proves that $\sup_y R_\lambda \varpi(y)<\infty$.

Therefore, we see
\begin{equation*}%\label{app-e25}
    \liminf_{\lambda \to\infty} \frac{R_\lambda \varpi(x)}{\sup_{y\in D } R_\lambda \varpi(y)}
    \geq \liminf_{\lambda \to\infty} \frac{R_\lambda \varpi(x)}{\sup_{y\in \rn} R_\lambda \varpi (y)}
    \geq  c>0,
\end{equation*}
and by Lemma~\ref{lem-reg} all points of $D$ are regular.
\qed

\subsection*{Proof of Lemma~\ref{lem-SK}}

The case  $\gamma=1$ is already contained in \cite{KK13b}. Therefore, we consider only $\gamma\in (0,1)$. Without loss of generality we may assume that $D$ is closed.

\medskip
\textbf{a)}
Under our assumptions the transition density $p_t(x,y)$ of the process $X$ satisfies \eqref{set-e35} and \eqref{set-e36} for $t\in (0,1]$.  Using \eqref{set-e36} and the scaling property of the subordinator \eqref{s4-e07}, we get for any $T\in (0,1]$
\begin{align*}
    \int_0^T \int_D  &p^{(\gamma)}_t(x,y) \,\varpi(dy) \,dt\\
    & \leq C \int_0^T  \int_D  \int_0^1 \rho_s^n \big(f_{\textrm{up}} (\rho_s  \cdot) *Q_s \big)(x-y) t^{-1/\gamma} \sigma_1^{(\gamma)}(t^{-1/\gamma} s) \,ds \,\varpi(dy) \,dt\\
    &\qquad\mbox{} + \int_0^T  \int_D  \int_1^\infty p_s(x,y) t^{-1/\gamma} \sigma_1^{(\gamma)}(t^{-1/\gamma} s) \,ds \,\varpi(dy) \,dt\\
    &=: C I_1(x,T) + I_2(x,T).
\end{align*}
We estimate $I_1(x,T)$ and $I_2(x,T)$ separately. For $I_2(x,T)$ we have
\begin{align*}
    I_2(x,T)
    &= \int_0^T \int_{t^{-1/\gamma}}^\infty  \int_D p_{\tau t^{1/\gamma}}(x,y) \sigma_1^{(\gamma)}(\tau) \,\varpi(dy) \,d\tau \,dt\\
    &= \int_{T^{-1/\gamma}}^\infty \int_{\tau^{-\gamma}}^T  \int_D p_{\tau t^{1/\gamma}}(x,y) \sigma_1^{(\gamma)}(\tau) \,\varpi(dy) \,dt \,d\tau\\
    &= \int_{T^{-1/\gamma}}^\infty \int_{1}^{\tau^\gamma T}  \int_D p_{v^{1/\gamma}}(x,y) \tau^{-\gamma}\sigma_1^{(\gamma)}(\tau) \,\varpi(dy) \,dv \,d\tau.
\end{align*}
Note that $\sup_{x,y\in \rn} p_t(x,y)\leq c$ for all $t\geq 1$. Indeed, since for $0<\epsilon <1$ we have $p_\epsilon (x,y)\leq C_\epsilon$ for all $x,y\in \rn$, the Chapman--Kolmogorov relation implies
$$
    p_t(x,y)=\int_\rn p_{t-\epsilon} (x,z) p_\epsilon(z,y) \,dz
    \leq C_\epsilon.
$$
Therefore,
\begin{align*}
    \sup_{x\in \rn} I_2(x,T)
    &\leq c_1 \int_{T^{-1/\gamma}}^\infty (\tau^\gamma T-1) \varpi(D)\tau^{-\gamma}\sigma_1^{(\gamma)}(\tau) \,d\tau\\
    &\leq c_1 \varpi(D) T \int_{T^{-1/\gamma}}^\infty \sigma_1^{(\gamma)}(\tau) \,d\tau
     \leq c_1 \varpi(D) T\xrightarrow[T\to 0]{} 0,
\end{align*}
where we used that $\int_0^\infty \sigma_1^{(\gamma)}(\tau)\,d\tau=1$.

For the first integral expression $I_1(x,T)$ we have
\begin{align*}
    I_1(x,T)
    &= \int_0^T \int_0^{t^{-1/\gamma}} \int_D \rho_{\tau t^{1/\gamma}}^n \big(f_{\textrm{up}} (\rho_{\tau t^{1/\gamma}}  \cdot) *Q_{\tau t^{1/\gamma}}  \big)(x-y) \sigma_1^{(\gamma)}(\tau) \,\varpi(dy) \,d\tau \,dt\\
    &= \int_0^{T^{-1/\gamma}}\int_0^T \int_D  \rho_{\tau t^{1/\gamma}}^n \big(f_{\textrm{up}} (\rho_{\tau t^{1/\gamma}}  \cdot) *Q_{\tau t^{1/\gamma}}  \big)(x-y)\sigma_1^{(\gamma)}(\tau) \,\varpi(dy) \,dt  \,d\tau\\
    &\qquad \mbox{}+ \int_{T^{-1/\gamma}}^\infty \int_0^{\tau^{-\gamma}}\int_D  \rho_{\tau t^{1/\gamma}}^n \big(f_{\textrm{up}} (\rho_{\tau t^{1/\gamma}}  \cdot) *Q_{\tau t^{1/\gamma}}  \big)(x-y) \sigma_1^{(\gamma)}(\tau) \,\varpi(dy) \,dt  \,d\tau\\
    &= \int_0^{T^{-1/\gamma}}\int_0^{\tau^\gamma T} \int_D  \rho_{v^{1/\gamma}}^n \big(f_{\textrm{up}} (\rho_{v^{1/\gamma}}  \cdot) *Q_{v^{1/\gamma}}  \big)(x-y)\tau^{-\gamma}\sigma_1^{(\gamma)}(\tau) \,\varpi(dy) \,dv  \,d\tau \\
    &\qquad \mbox{}+ \int_{T^{-1/\gamma}}^\infty \int_0^{1}\int_D  \rho_{v^{1/\gamma}}^n \big(f_{\textrm{up}} (\rho_{v^{1/\gamma}}  \cdot) *Q_{v^{1/\gamma}}  \big)(x-y) \tau^{-\gamma}\sigma_1^{(\gamma)}(\tau) \,\varpi(dy) \,dv  \,d\tau\\
    &=: I_{11}(x,T)+ I_{12}(x,T).
\end{align*}
For $I_{12}(x,D)$ we have
$$
    I_{12}(x,T)
    =\left[\int_{T^{-1/\gamma}}^\infty  \tau^{-\gamma}\sigma_1^{(\gamma)}(\tau) \,d\tau \right]
      \int_0^{1}\int_D  \rho_{v^{1/\gamma}}^n \big(f_{\textrm{up}} (\rho_{v^{1/\gamma}}  \cdot) *Q_{v^{1/\gamma}}  \big)(x-y) \,\varpi(dy) \,dv.
$$
Since $\lim_{T\to 0}\int_{T^{-1/\gamma}}^\infty  \tau^{-\gamma}\sigma_1^{(\gamma)}(\tau) \,d\tau = 0$, we get $\lim_{T\to 0} I_{12}(x,T) = 0$, if we can show that
\begin{equation}\label{app-e40}
    \sup_{x\in \rn} \int_0^{1}\int_D  \rho_{v^{1/\gamma}}^n \big(f_{\textrm{up}} (\rho_{v^{1/\gamma}}  \cdot) *Q_{v^{1/\gamma}}  \big)(x-y) \,\varpi(dy) \,dv
    < \infty
    \quad\text{for some $\gamma\in (\gamma_{\inf},1)$.}
 \end{equation}
Set $\ell := e_1 = (1,0,\dots 0)^\top$ and $\theta_t:= \inf\left\{ r\,:\, q^U(r\ell)\geq 1/t\right\}$. Because of \eqref{set-e30} we have  $\theta_t\asymp \rho_t$  for all $t\in (0,1]$.  Moreover, the mapping $r\mapsto q^U(r\ell)$ is absolutely continuous, and we have, cf.\ \cite{K13},
\begin{equation}\label{app-e45}
    q^U(r_2 \ell) - q^U(r_1 \ell)
    = 2\int_{r_1}^{r_2} \frac{q^L(v\ell)}{v}\,dv,
    \qquad 0<r_1<r_2.
\end{equation}
Thus, the above calculations  give
\begin{equation}\label{app-e50}
\begin{split}
    &\sup_{x\in \rn} \int_0^1\int_D \rho_{v^{1/\gamma}}^n \cdot \big(f_{\textrm{up}} (\rho_{v^{1/\gamma}}  \cdot) *Q_{v^{1/\gamma}}  \big)(x-y)\,\varpi(dy)\,dv \\
    &\leq c_1 \sup_{x\in \rn} \int_0^1 \int_D \theta_{v^{1/\gamma}}^n \cdot \big(f_{\textrm{up}} (c_2\theta_{v^{1/\gamma}} \cdot) * Q_{v^{1/\gamma}}\big)(x-y) \,\varpi(dy)\,dv\\
    &= c_1 a_3 \sup_{x\in \rn} \int_0^1 \int_\rn  \int_0^\infty  \theta_{v^{1/\gamma}}^n  \cdot \varpi \left\{ y\in D\,:\, e^{-c_2 a_4 |x-y-z| \theta_{v^{1/\gamma}} } \geq s\right\} \,ds\,Q_{v^{1/\gamma}}(dz)\,dv\\
    &= c_1 c_2 a_3 a_4 \sup_{x\in \rn} \int_0^1 \! \int_\rn \! \int_0^\infty  \theta_{v^{1/\gamma}}^n \!\cdot\! \varpi\left\{ y\in D : |x-y-z| \theta_{v^{1/\gamma}} \leq r\right\}  e^{-c_2 a_4 r} \,dr\,Q_{v^{1/\gamma}}(dz)\,dv\\
    &= C_1 \int_0^\infty \int_0^1 \int_\rn \sup_{x\in\rn} \frac{\varpi\left\{ y\in D \,:\, |x-y-z|\leq ur\right\}}{(q^U)^{\gamma}(\ell u^{-1})}  e^{-C_2 r} \,\tilde{Q}_{u}(dz)\,\frac{du}{u^{n+1}} \,dr\\
    &\leq  \kappa C_1  \int_0^\infty \int_0^1   \sup_{\xi\in \rn} \frac{\varpi\left\{ y\in D\,:\, |\xi-y|\leq ur\right\}}{(q^*)^{\gamma}(1/u)}  e^{-C_2 r} \,\frac{du}{u^{n+1}}\,dr,
\end{split}
\end{equation}
where
%$C_1= 2c_1c_2 a_3 a_4 \gamma$, $C_2=c_2 a_4$,
$\tilde{Q}_u(dz)=Q_{v^{1/\gamma}}(dz)$ under the change of variables $v= (q^U)^{-\gamma}(\ell u^{-1})$, which was done in the second line from below. Note that $\frac{1}{t}= q^U (\theta_t \ell)$, and that by \eqref{app-e45} and $q^U\asymp q^L$, one has
$$
    \frac{dv}{du}\asymp (q^U)^{-\gamma} (\ell u^{-1} ) u^{-1}.
$$
The constant $\kappa$ is from \eqref{set-e30}.

Let us estimate the integrals in the last line of \eqref{app-e50}. Without loss of generality we assume that $C_2=1$. Put
$h(r):= \sup_{\xi\in \rn} \varpi\left\{ y\in D \,:\, |\xi-y|\leq r\right\}$.
Split
\begin{align*}
    J
    := \left[\int_0^1\int_0^1 + \int_1^\infty\int_0^1\right] \frac{h(ur)}{(q^*)^{\gamma}(1/u)} \, \frac{du}{u^{n+1}} \, e^{-r}\,dr
    =: J_1 + J_2.
\end{align*}
From the  monotonicity of $h(r)$ and the assumption \eqref{s4-e15}
$$
    J_1
    \leq \int_0^1 \frac{h(u)}{(q^*)^{\gamma}(1/u)} \,\frac{du}{u^{n+1}}
    \cdot \int_0^1 e^{-r}\,dr
    <\infty.
$$
Using the monotonicity of $q^*$, we get
 \begin{align*}
    J_2
    &\leq c_1\int_1^\infty \left[\int_0^{r}  \frac{h(v)}{(q^*)^{\gamma}(1/v)} \,\frac{dv}{v^{n+1}} \right] r^n e^{-r}\,dr\\
    &= \left[ \int_1^\infty \int_0^1 + \int_1^\infty \int_1^r \right] \frac{h(v)}{(q^*)^{\gamma}(1/v)} \,r^n e^{-r}\, \frac{dv}{v^{n+1}} \,dr
     =: J_{21} + J_{22}.
 \end{align*}
 Clearly, $J_{21}<\infty$. For $J_{22}$ we have
\begin{align*}
    J_{22}
    &\leq \int_0^\infty \left[\int_{v}^\infty r^n e^{-r} \,dr\right]  \frac{h(v)}{(q^*)^{\gamma}(1/v)} \,\frac{dv}{v^{n+1}}\\
    &\leq c_2\int_0^\infty e^{-\epsilon v} \left[\int_{v}^\infty r^n e^{-(1-\epsilon)r}\, dr \right]  \frac{h(v)}{(q^*)^{\gamma}(1/v)} \,\frac{dv}{v^{n+1}}\\
    &\leq c_3 \int_0^\infty e^{-\epsilon v } \frac{h(v)}{(q^*)^{\gamma}(1/v)} \,\frac{dv}{v^{n+1}}.
 \end{align*}
By \eqref{s4-e15} and the fact that the integrand is bounded by $C e^{\epsilon v}$ for $v>1$ show that the integral in the last line is finite. Thus, \eqref{app-e40} holds true, implying that $\sup_{x\in \rn}I_{12}(x,T)\to 0$ as $T\to 0$.

Let us estimate $I_{11}(x,T)$. Define $\phi(u) := 1/\theta_u$,
$$
    I(v,\tau,T) := \I_{ \{\tau \leq T^{-1/\gamma}\}}  e^{-\epsilon v /(2 \phi( T^{1/\gamma}  \tau))},
$$
and recall that
$h(r):= \sup_{\xi\in \rn} \varpi\left\{ y\in D \,:\, |\xi-y|\leq r\right\}$.
By the same arguments as those which we have used in \eqref{app-e50}, we derive
\begin{align*}
    &\sup_{x\in \rn} I_{11}(x,T)\\
    &\leq c_1 \sup_{x\in \rn}\int_0^{T^{-1/\gamma}} \int_0^{T\tau^\gamma} \int_D  \theta_{v^{1/\gamma}}^n \cdot \big(f_{\textrm{up}} (c_2\theta_{v^{1/\gamma}}  \cdot) *Q_{v^{1/\gamma}}  \big)(x-y) \tau^{-\gamma}\sigma_1^{(\gamma)}(\tau) \,\varpi(dy) \,dv \,d\tau\\
    &\leq c_2  \int_0^{T^{-1/\gamma}}  \int_0^\infty \int_0^{\phi(T^{1/\gamma}\tau)} \frac{h(ur)}{(q^*)^{\gamma}(1/u)} \,e^{-c_3 r} \tau^{-\gamma}\sigma_1^{(\gamma)}(\tau)  \,\frac{du}{u^{n+1}} \,dr \,d\tau\\
    &\leq c_2 \int_0^{T^{-1/\gamma}}  \int_0^1 \int_0^{\phi(T^{1/\gamma}\tau)} \frac{h(u)}{(q^*)^{\gamma}(1/u)} \,e^{-c_3 r} \tau^{-\gamma}\sigma_1^{(\gamma)}(\tau) \,\frac{du}{u^{n+1}} \,dr \,d\tau\\
    &\qquad\mbox{}+ c_2 \int_0^{T^{-1/\gamma}} \int_1^\infty \int_0^{\phi(T^{1/\gamma}\tau)} \frac{h(ur)}{(q^*)^{\gamma}(1/u)} \,e^{-c_3 r} \tau^{-\gamma}\sigma_1^{(\gamma)}(\tau) \,\frac{du}{u^{n+1}} \,dr \,d\tau.
\end{align*}

The first term is estimated from above by
$$
    \int_0^{T^{-1/\gamma}} \int_0^{\phi(T^{1/\gamma}\tau)}
        \frac{h(u)}{(q^*)^{\gamma}(1/u)}\, \tau^{-\gamma}\sigma_1^{(\gamma)}(\tau)  \,\frac{du}{u^{n+1}}\, d\tau,
$$
which tends to zero as $T\to 0$ by the dominated convergence theorem.

For the second term we have for some $\epsilon>0$
\begin{align*}
    \int_0^{T^{-1/\gamma}} &\int_1^\infty \int_0^{r \phi(T^{1/\gamma}\tau)}
        \frac{h(v)}{(q^*)^{\gamma}(1/v)}\, r^n  e^{-c_3 r} \tau^{-\gamma}\sigma_1^{(\gamma)}(\tau)   \,\frac{dv}{v^{n+1}}\,dr\,d\tau\\
    &\leq  \int_0^{T^{-1/\gamma}}  \int_0^\infty \int_{v/\phi( T^{1/\gamma}  \tau)}^\infty
        \frac{h(v)}{(q^*)^{\gamma}(1/v)} \,r^n  e^{-c_3 r} \tau^{-\gamma}\sigma_1^{(\gamma)}(\tau) dr\, \frac{dv}{v^{n+1}} \,d\tau\\
    &\leq c_4 \int_0^{T^{-1/\gamma}}  \int_0^\infty e^{- \epsilon v/\phi( T^{1/\gamma}  \tau)}
    \frac{h(v)}{(q^*)^{\gamma}(1/v)} \tau^{-\gamma}\sigma_1^{(\gamma)}(\tau)  \,\frac{dv}{v^{n+1}} \,d\tau \\
    &\leq  c_4 \int_0^\infty \int_0^\infty
    \frac{ e^{-\epsilon v/2\phi(1)}h(v)}{(q^*)^{\gamma}(1/v)}\,  \tau^{-\gamma}\sigma_1^{(\gamma)}(\tau)  I(v,\tau,T)\,  \frac{dv}{v^{n+1}}\, d\tau.
\end{align*}
Note that $I(v,\tau, T)\leq 1$, and $\lim_{T\to 0}I(v,\tau,T)= 0$ a.e. From Euler's Gamma-integral
$s^{-a} = \Gamma(a)^{-1}\int_0^\infty e^{-s x}x^{a-1}\,dx$, $a>0$, we derive
\begin{align*}
    \int_0^\infty  s^{-a} \sigma_1^{(\gamma)}(s)\,ds
    = \int_0^\infty\int_0^\infty\frac{e^{-s x}x^{a-1}}{\Gamma(a)}\, \sigma_1^{(\gamma)}(s)\,ds\,dx
    = \int_0^\infty \frac{e^{-x^\gamma} x^{a-1}}{\Gamma(a)}\,dx
    =\frac{\Gamma(a/\gamma)}{\gamma\Gamma(a)}.
\end{align*}
By a dominated convergence argument, $\lim_{T\to 0}\sup_{x\in \rn} I_{11}(x,T)= 0$. This finishes the proof of a).

\bigskip
\textbf{b)} Without loss of generality we may assume that $T\in (0,1/2]$. Using \eqref{set-e35},  we have
\begin{align*}
    \int_0^T \int_D &p^{(\gamma)}(t,x,y) \,\varpi(dy) \,dt\\
    &\geq \int_0^T \int_D\int_0^\infty \rho_s^n f_{\textrm{low}}(\rho_s (x-y)) t^{-1/\gamma} \sigma_1^{(\gamma)} ( t^{-1/\gamma}s ) \,ds \,\varpi(dy) \,dt\\
    &\geq \int_0^T \int_D\int_0^1 \rho_s^n f_{\textrm{low}}(\rho_s (x-y))t^{-1/\gamma} \sigma_1^{(\gamma)}  ( t^{-1/\gamma}s ) \,ds \,\varpi(dy) \,dt \\
    &= \int_0^T \int_0^{t^{-1/\gamma}} \int_D \rho_{\tau t^{1/\gamma}}^n  f_{\textrm{low}} (\rho_{\tau t^{1/\gamma}}(x-y)) \sigma_1^{(\gamma)} (\tau) \,\varpi(dy) \,d\tau \,dt\\
    &= \int_0^{T^{-1/\gamma}} \int_0^T \int_D \rho_{\tau t^{1/\gamma}}^n f_{\textrm{low}} (\rho_{\tau t^{1/\gamma}}(x-y)) \sigma_1^{(\gamma)}  (\tau)  \,\varpi(dy) \,dt  \,d\tau\\
    &\qquad\mbox{} + \int_{T^{-1/\gamma}}^\infty \int_0^{\tau^{-\gamma}} \int_D  \rho_{\tau t^{1/\gamma}}^n f_{\textrm{low}} (\rho_{\tau t^{1/\gamma}}(x-y)) \sigma_1^{(\gamma)}(\tau) \,\varpi(dy) \,dt  \,d\tau\\
    &\geq \int_0^{T^{-1/\gamma}} \int_0^T \int_D \rho_{\tau t^{1/\gamma}}^n f_{\textrm{low}} (\rho_{\tau t^{1/\gamma}}(x-y))  \sigma_1^{(\gamma)}(\tau) \,\varpi(dy) \,dt  \,d\tau\\
    &= \int_0^{T^{-1/\gamma}} \int_0^{\tau^\gamma T} \int_D \rho_{v^{1/\gamma}}^n  f_{\textrm{low}} (\rho_{v^{1/\gamma}}(x-y))\tau^{-\gamma} \sigma_1^{(\gamma)}(\tau) \,\varpi(dy) \,dv  \,d\tau\\
    &\geq \int_1^{2^{1/\gamma}} \tau^{-\gamma} \sigma_1^{(\gamma)}(\tau) \,d\tau
    \int_0^T \int_D  \rho_{v^{1/\gamma}}^n  f_{\textrm{low}}(\rho_{v^{1/\gamma}}(x-y)) \,\varpi(dy) \,dv.
\end{align*}
Using the form of $f_{\textrm{low}}$ and the fact that $\rho_t \asymp \theta_t$, we see that the double integral is bounded from below by
$$
    \int_0^T\int_D  \theta_{v^{1/\gamma}}^n  f_{\textrm{low}} (c\theta_{v^{1/\gamma}}(x-y)) \,\varpi(dy) \,dv,
$$
where $c>0$ is some constant. Proceeding as in the estimate for  $I_{11}(x,T)$, we get for this expression
\begin{align*}
    \int_0^T\int_D  &\theta_{v^{1/\gamma}}^n  f_{\textrm{low}} (c_2\theta_{v^{1/\gamma}}(x-y)) \,\varpi(dy) \,dv\\
    &= \int_0^T \int_0^1  \theta_{v^{1/\gamma}}^n  \varpi\left\{ y\in D\,:\, c_2 d_2 \theta_{v^{1/\gamma}}|x-y|\leq r\right\} \,dr \,dv\\
    &\geq c_3 \int_0^{\phi(T^{1/\gamma})} \int_0^1 \frac{\varpi\left\{ y\in D\,:\, |x-y|\leq r u\right\}}{(q^*)^\gamma(1/u)}\,dr \,\frac{du}{u^{n+1}}\\
    &\geq \frac{c_3}2  \int_0^{\phi(T^{1/\gamma})} \frac{\varpi\left\{ y\in D \,:\, |x-y|\leq 2^{-1} u\right\}}{(q^*)^\gamma(1/u)}\, \frac{du}{u^{n+1}}.
\end{align*}
Combining everything, we have shown
\begin{equation}\label{app-e60}
    \int_0^T \int_D p^{(\gamma)}_t(x,y) \,\varpi(dy) \,dt
    \geq \textrm{const.} \int_0^{\phi(T^{1/\gamma})} \frac{\varpi\left\{ y\in D \,:\, |x-y|\leq 2^{-1} u\right\}}{(q^*)^\gamma(1/u)} \,\frac{du}{u^{n+1}}.
\end{equation}
Therefore, whenever $\varpi\in \Kato$  with respect to $p^{(\gamma)}_t(x,y)$, then   \eqref{s4-e20} holds true.
\qed

\end{document}